\documentclass[11pt]{amsart}

\usepackage[T1]{fontenc}
\usepackage[utf8]{inputenc}
\usepackage[english]{babel}
\usepackage{amsmath}
\usepackage{amssymb}
\usepackage{amsfonts}
\usepackage{bbm}
\usepackage{bigints}
\usepackage{graphicx}
\usepackage{microtype}
\usepackage{cite}
\usepackage[hidelinks]{hyperref}

\hypersetup{
    pdftitle={Lagrange interpolation processes based on the zeros of anti-Gauss Jacobi polynomials},
    pdfauthor={Patricia Díaz de Alba, Luisa Fermo, Valerio Loi, Donatella Occorsio},
    pdfkeywords={Lagrange interpolation, orthogonal polynomials, Gauss-Jacobi quadrature rule, anti-Gaussian formula}
}

\newcommand{\dx}{\,dx}
\def\C{\mathcal{C}}
\numberwithin{equation}{section}

\newtheorem{theorem}{Theorem}[section]
\newtheorem{lemma}[theorem]{Lemma}

\newtheorem{proposition}[theorem]{Proposition}

\theoremstyle{remark}
\newtheorem{remark}[theorem]{Remark}
\theoremstyle{definition}
\newtheorem{test}{Test}

\graphicspath{{figures/}}

\title[Lagrange interpolation at anti-Gauss Jacobi nodes]{Lagrange interpolation processes based on the zeros of anti-Gauss Jacobi polynomials}

\author[Patricia D\'iaz de Alba]{Patricia D\'iaz de Alba\textsuperscript{1}}
\author[Luisa Fermo]{Luisa Fermo\textsuperscript{2}}
\author[Valerio Loi]{Valerio Loi\textsuperscript{3}}
\author[Donatella Occorsio]{Donatella Occorsio\textsuperscript{4}}

\dedicatory{\normalfont
\textsuperscript{1}Department of Mathematics and Computer Science, University of Cagliari,
Via Ospedale 72, 09124, Italy; \texttt{patricia.diazda@unica.it}\\
\textsuperscript{2}Department of Mathematics and Computer Science, University of Cagliari,
Via Ospedale 72, 09124, Italy; \texttt{fermo@unica.it}\\
\textsuperscript{3}Department of Science and High Technology, Insubria University,
Via Valleggio 11, Como, 22100, Italy; \texttt{vloi@studenti.uninsubria.it}\\
\textsuperscript{4}Department of Mathematics and Computer Science, University of Basilicata,
Viale dell'Ateneo Lucano 10, 85100 Potenza, Italy and Istituto per le Applicazioni
del Calcolo ``Mauro Picone'', Naples branch, C.N.R. National Research Council of Italy,
Via P. Castellino, 111, 80131 Napoli, Italy; \texttt{donatella.occorsio@unibas.it}
}

\date{}

\keywords{Lagrange interpolation, orthogonal polynomials, Gauss--Jacobi quadrature rule, anti-Gaussian formula}
\subjclass[2020]{41A05, 41A10, 42C05}

\begin{document}

\begin{abstract}
This paper introduces and investigates a new  Lagrange interpolation process based on the zeros of anti-Gauss Jacobi polynomials. Fundamental properties of anti-Gauss nodes, including their asymptotic distribution, are established, together with estimates for the associated polynomials and their derivatives. These results provide the basis for the construction of an interpolation process whose weighted Lebesgue constants exhibit logarithmic growth, ensuring optimal approximation properties. Compared with previously known interpolation schemes based on Jacobi nodes, the proposed process  achieves optimal Lebesgue constants for a shifted range of endpoint weight parameters, allowing the use of smaller endpoint weight exponents. Convergence estimates are established for functions in suitable weighted Sobolev spaces, and numerical experiments support the theoretical findings.
\end{abstract}

\maketitle

\section{Introduction}
In recent years, anti-Gaussian formulas have attracted considerable attention. Introduced by Dirk Laurie in 1996, these are $(n+1)$-point quadrature formulas with the same degree of exactness as the $(n)$-point Gauss formula \cite{Laurie1}. They are termed \emph{anti-Gaussian} because of a distinctive property of their associated error: when applied to polynomials of degree up to $2n+1$, the error has the same magnitude as that of the $(n)$-point Gauss rule, but the opposite sign. This property makes it possible to construct more accurate and efficient quadrature schemes, such as the averaged formula, obtained by taking the average of the Gauss and anti-Gaussian rules.

Several schemes involving anti-Gaussian formulas have been developed \cite{Alqahtani2025,Djukic2023,drst2016,ReichelSpalevic2021,ReichelSpalevic2026,spalevic2007,spalevic2017} and successfully applied in a variety of contexts \cite{djukic2025}. Among these applications, we mention the numerical solution of second-kind integral equations on bounded intervals and on the positive real half-line \cite{DFR2020,DFR2025,djukic2024,FRRS2024}, where such schemes have proved particularly effective, especially in multidimensional settings; see, for instance, \cite{DFR2025,DFR2026,djukic2024}.

The quadrature nodes  of anti-Gaussian rules are the zeros of a polynomial $\tilde{p}_{n+1}$, which, from now on, we will refer to as the \emph{anti-Gauss polynomial}. This polynomial is defined in terms of the orthogonal polynomial $p_n$ associated with the corresponding $(n)$-point Gauss rule, and several of its fundamental properties were established in the original paper by Laurie \cite{Laurie1}. For a large class of weight functions, the zeros lie in the open integration interval and interlace with those of the corresponding polynomial $p_n$.

Anti-Gauss Laguerre polynomials have recently been studied in \cite{FO2026}. The authors derive estimates for the zeros, as well as for the polynomials and their derivatives. They also introduce a modified Lagrange interpolation polynomial and prove its convergence in suitable weighted Sobolev spaces.

In this paper, we focus on the \emph{anti-Gauss Jacobi polynomials}, which are defined in terms of the orthogonal polynomials associated with the Jacobi weight  $v^{\alpha,\beta}(x)=(1-x)^\alpha (1+x)^\beta$, $\alpha,\beta>-1$. Our main goal is to introduce an optimal Lagrange  interpolation process in the uniform
norm, based on the zeros of these polynomials, i.e., one for which the corresponding Lebesgue constants have a logarithmic behavior in the uniform norm. Achieving this requires two key ingredients: the arc-sine distribution of the set of nodes and the uniform boundedness of the sequence of considered polynomials \cite{MO2001}. 
In this paper, we prove that the zeros of anti-Gauss Jacobi polynomials are arc-sine distributed whenever they are internal. Moreover, we show that in suitable weighted spaces $C_u$ equipped with a uniform metric, the sequence of anti-Gauss Jacobi polynomials $\tilde{p}_{n+1}$ is uniformly bounded, i.e., $\sup_n \|\tilde{p}_{n+1}\|_{C_u}<\infty$, under suitable relations of  the weights $v^{\alpha,\beta}$ and $u$.  As a consequence, we establish the optimality of the corresponding Lebesgue constants and derive convergence and error estimates for the proposed interpolation process. We remark that a similar result holds for Lagrange interpolation at the zeros of the Jacobi polynomials. However, the proposed interpolation scheme achieves the optimal growth of the Lebesgue constants for a different range of endpoint weight parameters, allowing smaller endpoint weight exponents and, in particular, including the unweighted case when $\alpha,\beta\leq 1/2$. As a consequence, the two interpolation processes complement each other, yielding optimal Lebesgue constant estimates over a wider range of the weight exponents  $\gamma,\delta$. We recall that the problem of constructing optimal interpolation processes has attracted considerable attention, and several optimal Lagrange interpolation processes have been developed in recent years; see, for instance, \cite{Mastroianni08, MO2001, MR1997} and the references therein. Indeed, Lagrange interpolation plays a central role in numerical analysis and approximation theory. Besides providing an effective tool for function approximation, it constitutes a fundamental ingredient in the construction of quadrature formulas, projection methods, and numerical schemes for differential and integral equations. Its widespread use stems from both its ease of implementation and its solid theoretical foundation.

The paper is organized as follows. Section~\ref{sec:preliminaries} introduces the function spaces and summarizes some well-known results on weighted polynomial approximation, Jacobi polynomials, and anti-Gauss Jacobi polynomials. Section~\ref{sec:newproperties} presents new results on the distribution of the zeros, together with estimates for the polynomials and their derivatives. In Section~\ref{sec:interpolation}, we introduce a Lagrange interpolation polynomial based on the zeros of the anti-Gauss Jacobi polynomials. Section~\ref{sec:tests} reports several numerical experiments illustrating the behavior of the Lebesgue constants and the convergence rates achieved in function approximation, and compares the results with those obtained by the corresponding Lagrange interpolation processes based on the zeros of the Jacobi polynomials. All proofs are collected in Section~\ref{sec:proofs}. Finally, Section~\ref{sec:conclusions} concludes the paper with some final remarks and discusses possible directions for future research.

\section{Preliminaries}\label{sec:preliminaries}
Here and in the sequel, $\C$ denotes a positive constant which may vary in different formulas. We write $\C =\C(a,b,\dots)$ to indicate that $\C$ depends on the parameters $a,b,\dots$, and $\C \neq \C(a,b,\dots)$ to state that it is independent of them. Moreover, if $A(n,f,x,\dots)$ and $B(n,f,x,\dots)$ are two positive functions depending on certain variables $n,f,x, \dots$, then we write $A \sim B$ if there exist a positive constant $\C \neq \C(n,f,x,\dots)$ such that the inequalities
$\C^{-1} B(n,f,x,\dots) \leq A(n,f,x,\dots) \leq \C B(n,f,x,\dots)$ are satisfied.
Moreover, for any $\rho,\sigma\in\mathbb{R}$, as usual, we set
$$v^{\rho,\sigma}(x):=(1-x)^\rho(1+x)^\sigma,\qquad  x\in (-1,1).$$
For any $z\in \mathbb{R}$ the notation  $(z)_n$ denotes  the Pochhammer symbol,  defined for $n \in \mathbb{N}$ by
$(z)_0:=1$, $\displaystyle (z)_n:=\prod_{k=1}^n (z+k-1)$. 

\subsection{Function spaces}
Let $u$ be the Jacobi weight defined by
\begin{equation}\label{u}
u(x)=v^{\gamma,\delta}(x),\qquad x\in[-1,1], \quad \gamma,\delta\ge0.
\end{equation}
We denote by $C_u$ the Banach space of functions that are locally continuous on $(-1,1)$ and satisfy
\begin{equation}\label{lim}
\lim_{x\rightarrow +1}f(x)u(x)=0,\quad \mbox{if }\gamma>0,
\qquad\mbox{and}\qquad
\lim_{x\rightarrow -1}f(x)u(x)=0,\quad \mbox{if }\delta>0,
\end{equation}
endowed with the norm
\[
\|f\|_{C_u}:=\|fu\|_\infty=\max_{x\in[-1,1]}|f(x)|u(x).
\]

The conditions in \eqref{lim} are necessary to guarantee that (see, for instance,~\cite{Mastroianni08})
\[
\lim_{n\rightarrow\infty}E_n(f)_u=0,\qquad \forall\, f\in C_u,
\]
where $\mathbb{P}_n$ denotes the space of algebraic polynomials of degree at most $n$, and
\begin{equation}\label{best}
E_n(f)_u:=\inf_{P\in\mathbb{P}_n}\|f-P\|_{C_u}
\end{equation}
is the error of best approximation of $f$ in the weighted norm.

Moreover, the space $C_u$ can be characterized in terms of the error of best approximation. Indeed, it is shown in~\cite{ditzian2012} that
\[
\lim_{n\to\infty}E_n(f)_u=0
\quad\Longleftrightarrow\quad
f\in C_u.
\]
For smoother functions, let us denote by $W_r(u)$ the Sobolev-type subspace of $C_u$ of order $r \in \mathbb{N}$, $r \geq 1$ defined as 
\begin{equation*}
W_r(u)=\left\{f \in C_u : f^{(r-1)} \in AC((-1,1)), \|f^{(r)}\varphi^r u\|_\infty < \infty \right\},
\end{equation*}
where $\varphi(x)=\sqrt{1-x^2}$ and $AC((-1,1))$ denotes the collection of all functions which are absolutely continuous on every closed subset of $(-1,1)$, equipped with  norm 
$$\|f\|_{W_r(u)}:=\|fu\|_\infty+\|f^{(r)}\varphi^r u\|_\infty.$$

The following Favard inequality provides an estimate of $E_n(f)_u$ for functions $f\in W_r(u)$
\begin{equation}\label{stimaWr}
E_n(f)_{u} \leq \frac{\C}{n^r} \|f\|_{W_r(u)},  \quad \forall f \in W_r(u),\quad \C\neq \C(n,f).
\end{equation}

\subsection{Jacobi polynomials}
Let $v^{\alpha,\beta}$ be a Jacobi weight  of exponents  $\alpha,\beta>-1,$ and let $\{p_n\}_{n=0}^\infty$ be the corresponding sequence of monic orthogonal
polynomials on $[-1,\,1]$,  i.e.,
\begin{equation}\label{scalar}
\langle p_n, p_m \rangle_{v^{\alpha,\beta}} =\int_{-1}^1 p_n(x) p_m(x) v^{\alpha,\beta}(x) \dx =
\begin{cases}
0, & m \neq n, \\
c_n, & m=n,
\end{cases}
\end{equation}
where  
\begin{equation}\label{cn}
c_n=\frac{2^{2n+\alpha+\beta+1}}{2n+\alpha+\beta+1} \cdot 
\frac{\Gamma(n+\alpha+1)\Gamma(n+\beta+1)}{n!\,\Gamma(n+\alpha+\beta+1)} \left[ \frac{n!}{(n+\alpha+\beta+1)_n}\right]^2,
\end{equation}
and $\Gamma$ is the Gamma function.

Such a sequence satisfies the following three-term recurrence relation~\cite{Gautschi04}
\begin{equation}\label{recurrence}
\begin{cases}
p_{-1}(x)=0, \quad p_0(x)=1, \\
p_{n+1}(x)=(x-a_n) p_n(x)- b_n p_{n-1}(x), \quad n=0,1,2,\ldots,
\end{cases}
\end{equation}
where the coefficients $a_n$ and $b_n$ are given by
\begin{align}
a_n&= \frac{\beta^2-\alpha^2}{(2n+\alpha+\beta)(2n+\alpha+\beta+2)}, 
&n \geq 0, \label{aj}\\
b_0 &= \frac{2^{\alpha+\beta+1} \Gamma(\alpha+1) \Gamma(\beta+1)}{\Gamma(\alpha+\beta+2)}, \label{b0} \\
b_n&= \frac{4n(n+\alpha)(n+\beta)(n+\alpha+\beta)}{(2n+\alpha+\beta)^2
((2n+\alpha+\beta)^2-1)}, &n \geq 1. \label{bj}
\end{align}
Moreover, the values of $p_n$ at the endpoints $\pm 1$ depend on the degree $n$ and the parameters $\alpha$ and $\beta$, and are explicitly given by~\cite{JMN}
\begin{equation}\label{valuespol}
p_n(-1)=\frac{(-2)^n (\beta+1)_n}{(n+\alpha+\beta+1)_n}, \quad p_n(1)=\frac{2^n (\alpha+1)_n}{(n+\alpha+\beta+1)_n}.
\end{equation}

Let us denote by $\{x_{n,k}\}_{k=1}^n$ the zeros of $p_n$ in increasing order
$$-1=:x_{n,0}<x_{n,1}<\dots<x_{n,n}<x_{n,n+1}:=1.$$
It is well known that they are arc-sine distributed, i.e., 
$$\theta_{n,k}-\theta_{n,k+1} \sim \frac{\C}{n}, \quad k=0,1,\dots,n,$$
where $x_{n,k}=\cos{\theta_{n,k}}$,  $\theta_{n,k} \in [0,\pi]$.   
In a nutshell, 
\begin{equation}\label{gaussdist}
\Delta x_{n,k}:=x_{n,k+1}-x_{n,k} \sim \frac{\sqrt{1-x_{n,k}^2}}{n}, \quad k=0,\dots,n.
\end{equation} 
Concerning the computation of the zeros $\{x_{n,k}\}_{k=1}^n$, in 1962, Wilf~\cite{wilf1962} found that they are the eigenvalues of the following matrix, known as the Jacobi matrix, 
\begin{equation}\label{Jn}
J_n=
\begin{bmatrix}
a_0 & \sqrt{b_1} \\
\sqrt{b_1} & a_1 & \sqrt{b_2} \\
& \sqrt{b_2} & a_2 & \ddots \\ 
& & \ddots & \ddots & \sqrt{b_{n-1}} \\
& & & \sqrt{b_{n-1}} &  a_{n-1} \\
\end{bmatrix},
\end{equation}
where the entries $a_j$ and $b_j$ are defined in \eqref{aj} and \eqref{bj}, respectively.

For our aims, let us also introduce the sequence $\{p_n(v^{\alpha,\beta})\}_{n=0}^\infty$ of orthonormal polynomials  corresponding to the weight $v^{\alpha,\beta}$ having  positive leading coefficients, i.e.,
$$p_n(v^{\alpha,\beta},x) = \gamma_n x^n + \textrm{lower degree terms}, \quad \gamma_n>0,$$
 with 
\begin{equation*}
p_{n}(v^{\alpha,\beta},x) = \frac{p_n(x)}{\sqrt{c_n}},
\end{equation*}
where $c_n$  given in \eqref{cn}. Such polynomials satisfy the following~\cite{Nevai1979,szego1975,mr2009} 
\begin{equation*}
|p_{n}(v^{\alpha,\beta},x)| \sim \dfrac{\C}{\left(\sqrt{1-x}+\frac{1}{n}\right)^{\alpha+\frac{1}{2}}\left(\sqrt{1+x}+\frac{1}{n}\right)^{\beta+\frac{1}{2}}}, \qquad |x|\leq 1,
\end{equation*}
from which one can deduce (see, for instance,~\cite{szego1975,Nevai1979,Mastroianni08})
\begin{equation}\label{stimapol}
|p_{n}(v^{\alpha,\beta},x)| \sim \C 
\begin{cases}
 \dfrac{1}{\sqrt{v^{\alpha,\beta}(x) \varphi(x)}}  & x \in I_n:=\left[-1+\C n^{-2},1-\C n^{-2}\right], \\
 n^{\beta+\frac{1}{2}}  &x \in  \left[-1, -1+\C n^{-2} \right],\\
 n^{\alpha+\frac{1}{2}} & x \in \left[ 1-\C n^{-2},1\right],
\end{cases}
\end{equation}
where $\varphi(x)=\sqrt{1-x^2}$ and $\C(\alpha,\beta)=\C \neq \C(n)$.

Finally, for our aims let us recall the so-called Christoffel numbers~\cite{Mastroianni08}
\begin{equation} \label{lambda}
\lambda_{n,k}=
\left[ \sum_{k=0}^{n-1} [p_k(v^{\alpha,\beta},x_{n,k})]^2 \right]^{-1} = \left[ K_{n-1}(v^{\alpha,\beta},x_{n,k},x_{n,k}) \right]^{-1}
\end{equation}
where $K_{n-1}(v^{\alpha,\beta})$ is the Christoffel-Darboux kernel defined as 
\begin{equation}\label{ChrisDarb}
K_{n-1}(v^{\alpha,\beta},x,y):=
\sum_{j=0}^{n-1}p_j(v^{\alpha,\beta},x)p_j(v^{\alpha,\beta},y).
\end{equation}

Concerning the computation of the Christoffel numbers, Wilf~\cite{wilf1962}  proved that  they are strictly related to the matrix $J_n$. In fact, $\lambda_{n,k}=b_0 v^2_{k,1}$, where $b_0$ is given in \eqref{b0} and $v_{k,1}$ is the first component  of the normalized eigenvector corresponding to the eigenvalue $x_{n,k}$.

Furthermore, let us mention the following estimate~\cite{Nevai1979,Mastroianni08}
\begin{equation}\label{estimatelambda}
\lambda_{n,k} \sim v^{\alpha,\beta}(x_{n,k}) \frac{\sqrt{1-x_{n,k}^2}}{n}, \qquad k=1,\dots,n,
\end{equation}
where the constants in $\sim$ do not depend on $n$ and $k$.

\subsection{Anti-Gauss Jacobi polynomials}
In~\cite{Laurie1}, Laurie introduced the sequence of polynomials
$\{\tilde{p}_n\}_{n}$
defined by
\begin{equation}\label{recurrence_anti2}
\tilde{p}_{n+1}(x)=p_{n+1}(x)- b_n p_{n-1}(x), \qquad n=0,1,\dots,
\end{equation}
where $p_n$ denotes the $n$th orthogonal polynomial defined by~\eqref{recurrence} and $b_n$ is given by~\eqref{bj}. These polynomials were introduced to construct a quadrature scheme for estimating the error $e_n(f)$ of the $n$-point Gauss--Jacobi quadrature formula $G_n(f)$
\begin{equation*}
\int_{-1}^1 f(x) v^{\alpha,\beta}(x)  dx= \sum_{k=1}^{n} \lambda_{n,k} f(x_{n,k}) + e_{n}(f)=:G_n(f)+e_n(f),
\end{equation*} 
where the Christoffel number $\lambda_{n,k}$ are defined in \eqref{lambda}.
He then introduced a new quadrature formula based on $n+1$ nodes, \begin{equation}\label{anti-Gauss} 
\tilde{G}_{n+1}(f) = \sum_{k=1}^{n+1} \tilde{\lambda}_{n+1,k} f(\tilde{x}_{n+1,k}), 
\end{equation} 
whose error, denoted by $\tilde{e}_{n+1}(f)$, has the same magnitude as the Gauss error $e_n(f)$ but opposite sign,  for all polynomials of degree at most $2n+1$, i.e., 
\begin{equation}\label{error} \tilde{e}_{n+1}(f) = -e_n(f), \quad \text{for all } f \in \mathbb{P}_{2n+1}. 
\end{equation} 
In view of this property, he called the new formula the \emph{anti-Gaussian rule}. Moreover, he proved that relation \eqref{error} holds if and only if the quadrature nodes $\{\tilde{x}_{n+1,k}\}_{k=1}^{n+1}$ in \eqref{anti-Gauss} are the zeros of the polynomial $\tilde{p}_{n+1}$ defined by the recurrence relation \eqref{recurrence_anti2}. Motivated by their role in the anti-Gaussian rule, we shall refer to these polynomials as the \emph{anti-Gauss Jacobi polynomials}.

Several properties of the zeros $\{\tilde{x}_{n+1,k}\}_{k=1}^{n+1}$ are well established.
First, the zeros of $\tilde{p}_{n+1}$ coincide with the eigenvalues of the modified Jacobi matrix of order $n+1$,
\begin{equation}\label{Jn+1}
\tilde{J}_{n+1}=
\begin{bmatrix}
J_n & \sqrt{2 b_n}\mathbf{e}_n \\
\sqrt{2 b_n}\mathbf{e}_n^T & a_n
\end{bmatrix},
\end{equation}
where $\mathbf{e}_n=(0,0,\dots,1)^T\in\mathbb{R}^n$, while $J_n$, $a_n$, and $b_n$ are defined in
\eqref{Jn}, \eqref{aj}, and \eqref{bj}, respectively.

Furthermore, the zeros of $\tilde{p}_{n+1}$ interlace with the Gauss nodes
$\{x_{n,k}\}_{k=1}^n$, namely,
\begin{equation}\label{interlacing}
\tilde{x}_{n+1,1}<x_{n,1}<\tilde{x}_{n+1,2}<\cdots<\tilde{x}_{n+1,n}<x_{n,n}<\tilde{x}_{n+1,n+1}.
\end{equation}
As a consequence, the interior anti-Gauss nodes
$\tilde{x}_{n+1,k}$, $k=2,\ldots,n$, all lie to the interval $(-1,1)$, whereas only the first and the last nodes may lie outside it.  
More specifically, Dirk Laurie~\cite[Theorem 4]{Laurie1} (see also its reformulation in~\cite{DFR2026}) proved that the smallest anti-Gauss node,
$\tilde{x}_{n+1,1}$, belongs to the interval $[-1,1]$ if and only if
\begin{equation}\label{cond1}
(2\beta+1) n^2+(2\beta+1) (\alpha+\beta+1)n+\frac{1}{2}(\beta+1)(\alpha+\beta)(\alpha+\beta+1) \geq 0,
\end{equation}
whereas the largest node, $\tilde{x}_{n+1,n+1}$, belongs to $[-1,1]$ if and only if
\begin{equation}\label{cond2}
(2\alpha+1) n^2+(2\alpha+1) (\alpha+\beta+1)n+\frac{1}{2}(\alpha+1)(\alpha+\beta)(\alpha+\beta+1) \geq 0.
\end{equation} 

Moreover, the node $\tilde{x}_{n+1,1}$ coincides with the endpoint $-1$ if and only if $\beta=-1/2$ and $\alpha= \pm 1/2$, while the node $\tilde{x}_{n+1,n+1}$ coincides with the endpoint $1$ if and only if $\alpha=-1/2$ and $\beta=\pm 1/2$. 
Finally, in the Gegenbauer or ultraspherical case, namely when $\alpha=\beta$, all the anti-Gauss nodes lie in the open interval $(-1,1)$ if and only if $\alpha=\beta > -1/2$.

Throughout the paper, whenever interpolation on $C_u$ is considered, we assume that the parameters $\alpha$ and $\beta$ satisfy the strict inequalities in \eqref{cond1} and \eqref{cond2}, so that all the anti-Gauss nodes lie in $(-1,1)$.

Combining \eqref{recurrence_anti2} with the recurrence relation \eqref{recurrence}, we obtain
\begin{equation}\label{recurrence_anti1}
\tilde{p}_{n+1}(x)=2p_{n+1}(x)+(a_n-x)p_n(x).
\end{equation}
Evaluating \eqref{recurrence_anti1} at $x=\pm1$ and using \eqref{valuespol}, we find
\begin{align*}
\tilde{p}_{n+1}(-1)&=
\left[-\frac{4(n+\beta+1)(n+\alpha+\beta+1)}{(2n+\alpha+\beta+1)(2n+\alpha+\beta+2)}+a_n+1\right]p_n(-1),\\
\tilde{p}_{n+1}(1)&=
\left[\frac{4(n+\alpha+1)(n+\alpha+\beta+1)}{(2n+\alpha+\beta+1)(2n+\alpha+\beta+2)}+a_n-1\right]p_n(1),
\end{align*}
that is, 
\begin{align*}
\frac{\tilde{p}_{n+1}(-1)}{p_n(-1)}&=
-\frac{ \frac{\beta}{n}+\frac{1}{2n} +O(\frac{1}{n^2})}{1+\frac{2\alpha+2\beta+3}{2n}+\frac{(\alpha+\beta+1)(\alpha+\beta+2)}{4n^2} }+a_n ,\\
\frac{\tilde{p}_{n+1}(1)}{p_n(1)}&=
\frac{ \frac{\alpha}{n}+\frac{1}{2n} +O(\frac{1}{n^2})}{1+\frac{2\alpha+2\beta+3}{2n}+\frac{(\alpha+\beta+1)(\alpha+\beta+2)}{4n^2} }+a_n.
\end{align*}

Consequently, since $a_n=O(n^{-2})$, it follows that
\begin{equation}\label{valuepol_anti2}
\frac{\tilde{p}_{n+1}(-1)}{p_n(-1)}=-\frac{\beta+\frac12}{n}+O(n^{-2}), \qquad \frac{\tilde{p}_{n+1}(1)}{p_n(1)}=\frac{\alpha+\frac12}{n}+O(n^{-2}).
\end{equation}

\noindent Moreover,
\begin{align*}
\tilde{p}_{n+1}(x)& = 2 \sqrt{c_{n+1}} p_{n+1}(v^{\alpha,\beta},x)+ \sqrt{c_n} (a_n-x) p_{n}(v^{\alpha,\beta}, x)\\
& = \sqrt{c_n} \left[2\rho_n \,  p_{n+1}(v^{\alpha,\beta},x)+(a_n-x) p_{n}(v^{\alpha,\beta}, x)\right],
\end{align*}
where
\begin{equation*}
\rho_n:=\sqrt{\frac{c_{n+1}}{c_n}}=\sqrt{b_{n+1}}\sim \frac12 .
\end{equation*}
The identity $c_{n+1}=b_{n+1}c_n$ is the standard norm relation for monic
orthogonal polynomials.  Hence the normalized anti-Gauss polynomials have the form
\begin{equation*}
\frac{\tilde p_{n+1}(x)}{\sqrt{c_n}}
=
2\rho_n p_{n+1}(v^{\alpha,\beta},x)+(a_n-x)p_n(v^{\alpha,\beta},x),
\qquad 2\rho_n=1+O(n^{-2}).
\end{equation*}
This is the form used below for the local estimates. 
Furthermore, by virtue of~\cite[Section~5.1]{JMN},
\begin{equation}\label{derpn}
p'_n(v^{\alpha,\beta},x)
=\sqrt{n(n+\alpha+\beta+1)}\,
p_{n-1}(v^{\alpha+1,\beta+1},x),
\end{equation}
it follows from \eqref{recurrence_anti1} that
\begin{align}\label{der_anti-pol}
\frac{|\tilde p'_{n+1}(x)|}{\sqrt{c_n}}
&\le \C n\Bigl(
|p_n(v^{\alpha+1,\beta+1},x)|
+|p_{n-1}(v^{\alpha+1,\beta+1},x)|
\Bigr)
+|p_n(v^{\alpha,\beta},x)|,
\end{align}
where the constant $\C$ is independent of $n$ and $x$.

We conclude this section by recalling that the coefficients  of the anti-Gaussian rule, $\{\tilde{\lambda}_{n+1,k}\}_{k=1}^{n+1}$, can also be characterized in terms of the eigenvalue problem \eqref{Jn+1}. In fact,  
they are given by
\begin{equation}\label{lambdatilde}
\tilde{\lambda}_{n+1,k}=b_0 \, \tilde{v}^2_{k,1},
\end{equation}
where $b_0$ is defined by \eqref{b0} and $\tilde{v}_{k,1}$ is the first
component of the eigenvector associated to the eigenvalue $\tilde{x}_{n+1,k}$. By \eqref{lambdatilde}, one can deduce that the weights
$\{\tilde{\lambda}_{n+1,k}\}_{k=1}^{n+1}$ are strictly positive. Moreover from~\cite{Notaris2018}, 
denoting by  $\| \cdot\|_{v^{\alpha,\beta}}$  the norm induced by the scalar product \eqref{scalar}, one has the following identity
\begin{equation*}
\tilde{\lambda}_{n+1,k}= \frac{2\|p_n(x)\|^2_{v^{\alpha,\beta}}}{p_n(\tilde{x}_{n+1,k}) \, \tilde{p}'_{n+1}(\tilde{x}_{n+1,k})},
\end{equation*}
equivalent to 
\begin{equation}\label{lambatilde2}
\tilde{\lambda}_{n+1,k}=\frac{2\sqrt{c_n}}{p_n(v^{\alpha,\beta},\tilde{x}_{n+1,k}) \, \tilde{p}'_{n+1}(\tilde{x}_{n+1,k})}.
\end{equation}

\section{Some new properties and estimates of anti-Gauss Jacobi polynomials}\label{sec:newproperties}
From now on, to simplify the notation, we will set 
\begin{equation*}
\begin{aligned}
x_k&:=x_{n,k},  &&\lambda_{k}:={\lambda}_{n,k},  &&k=1,2,\dots,n, \\
\tilde{x}_{k}&:=\tilde{x}_{n+1,k}, &&\tilde{\lambda} _{k}:=\tilde{\lambda}_{n+1,k},   &&k=1,2,\dots,n+1.
\end{aligned}
\end{equation*}

In the next proposition, we investigate the distances of the first and last zeros of $\tilde p_{n+1}$ from the endpoints $-1$ and $1$, respectively, under the assumption that  $\tilde{x}_1,\tilde{x}_{n+1}\in(-1,1)$, i.e., when conditions \eqref{cond1} and \eqref{cond2} are satisfied strictly. The obtained estimates are analogous to those for Jacobi polynomials, where the distances of the extreme zeros of $p_n(v^{\alpha,\beta})$ from the endpoints $\pm1$ are of order $n^{-2}$
\medskip
\begin{proposition}\label{prop:nodes1}
Assume that the parameters $\alpha$ and $\beta$ satisfy the strict  inequalities   in \eqref{cond1} and \eqref{cond2}. Then, $$ \tilde{x}_k \in I_n, \qquad \forall k=1,\dots,n+1, \quad \textit{where} \quad I_n=\left[-1+\C n^{-2},1-\C n^{-2}\right].$$
\end{proposition}

The next result shows that  the distribution of zeros $\tilde{x}_k$ follows the arc-sine law, that is, its density increases as one approaches the endpoints. 
\medskip
\begin{proposition}\label{prop:nodes3}
Assume that the parameters $\alpha$ and $\beta$ satisfy the strict  inequalities   in \eqref{cond1} and \eqref{cond2}.   Then, the nodes $\tilde{x}_k \in (-1,1)$ with $k=1,\dots,n+1$ have the same asymptotic arc-sine distribution as the zeros of the Jacobi polynomial $p_n(v^{\alpha,\beta})$.

Moreover, with $\varphi(x)=\sqrt{1-x^2}$, one has 
\begin{equation}\label{distanza2}
\Delta \tilde x_{k}:=\tilde x_{k+1}-\tilde x_{k}  \sim \dfrac{\varphi(\tilde x_{k})}{n}, \quad k=1,\dots,n,
\end{equation}
where the constants in $\sim$ are independent of $n$ and $k$.
\end{proposition}

The asymptotic arc-sine distribution described in Proposition~\ref{prop:nodes3} is illustrated in Figure~\ref{fig:arcsindistr}, where the Gauss nodes $x_k$ and anti-Gauss nodes $\tilde{x}_k$ are displayed for $n=25$ and two different Jacobi weights.

\begin{figure}[ht]
\centering
\includegraphics[width=0.48\textwidth]{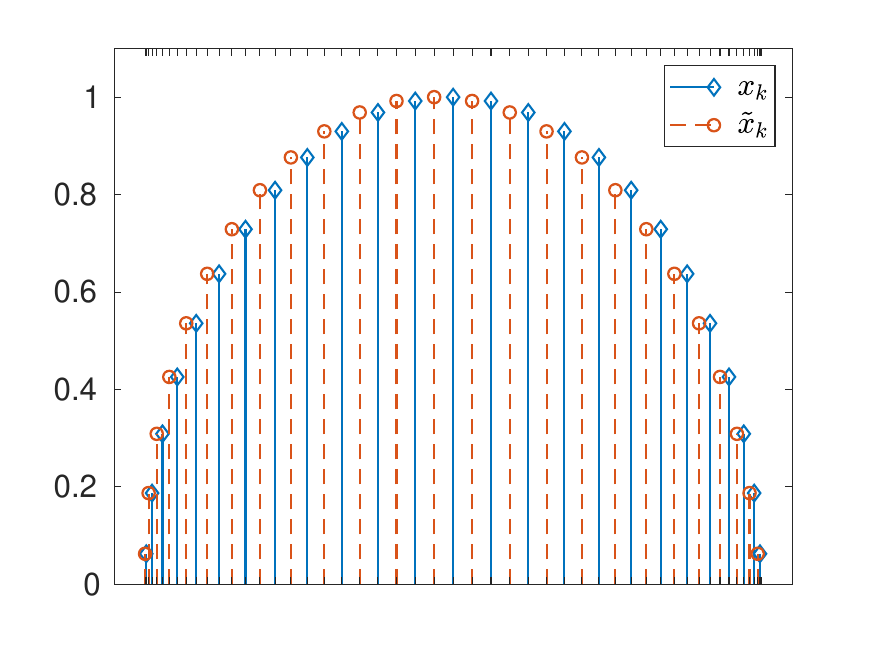}
\includegraphics[width=0.48\textwidth]{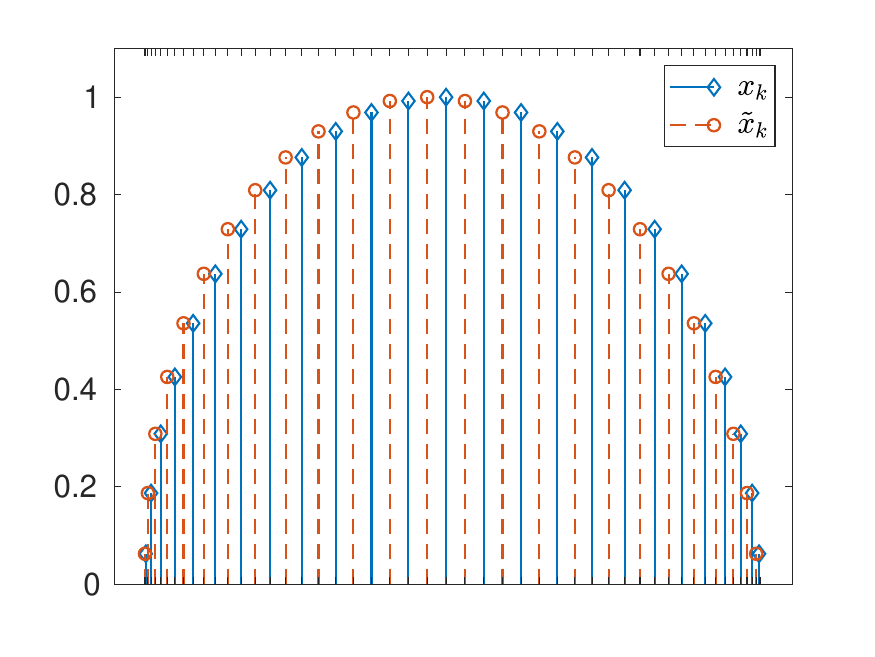}
\caption{Arc-sine distribution of Gauss (continuous line) and anti-Gauss (dashed line) nodes for $n=25$. On the left: $\alpha=\beta=0$; on the right: $\alpha=\frac{1}{2}$ and $\beta=-\frac{1}{4}$.}
\label{fig:arcsindistr}
\end{figure}

\begin{remark}
Note that by combining Proposition~\ref{prop:nodes1} and~\ref{prop:nodes3}, we can assert
\begin{equation}\label{distanza}
\Delta \tilde x_{k}  \sim \dfrac{\varphi(\tilde x_{k})}{n}, \quad k=0,1,\dots,n+1,
\end{equation}
where 
\begin{equation}\label{defdistanza}
\Delta \tilde x_{0}=-1+\tilde x_1, \qquad \Delta \tilde x_{k}:=\tilde x_{k+1}-\tilde x_{k}, \quad k=1,\dots,n, \qquad \Delta \tilde x_{n+1}=1-\tilde{x}_{n+1}.
\end{equation}
\end{remark}

Proposition~\ref{prop:nodes3} allows us to relax the hypothesis of Lemma 1 in~\cite{DFR2020} that can be reformulated as follows.

\medskip
\begin{proposition}
Assume that the parameters $\alpha$ and $\beta$ satisfy the strict  inequalities   in \eqref{cond1} and \eqref{cond2},  and let 
$\{{\tilde \lambda}_{k}\}_{k=1}^{n+1}$ be defined in \eqref{lambatilde2}. 
Then,
\begin{equation*} 
\tilde{\lambda}_k \sim 
v^{\alpha,\beta}(\tilde{x}_k) \, \varphi(\tilde{x}_k) \, \Delta \tilde{x}_k,
\qquad k=1,\ldots,n+1,
\end{equation*}
where the constants in $\sim$ are independent of $n$ and $k$, and $\Delta \tilde{x}_k$ are defined as in \eqref{defdistanza}.
\end{proposition}

Let us now introduce the normalized anti-Gauss Jacobi polynomial
\begin{equation*} 
\widehat p_{n+1}(x):=
\frac{\tilde p_{n+1}(x)}{\| \tilde p_{n+1}\|_{v^{\alpha,\beta}}}
=
\frac{\tilde p_{n+1}(x)}{\sqrt{c_{n+1}+b_n^2c_{n-1}}},
\end{equation*}
where the equality $\| \tilde p_{n+1}\|_{v^{\alpha,\beta}}=\sqrt{c_{n+1}+b_n^2c_{n-1}}$ follows from the identity $$\tilde p_{n+1}=p_{n+1}-b_np_{n-1}$$ and the
orthogonality of $p_{n+1}$ and $p_{n-1}$. 
In particular, since $c_{n+1}=b_{n+1}c_n$ and $b_n\to1/4$, we have
\begin{equation*}
\|\tilde p_{n+1}\|_{v^{\alpha,\beta}}^2
=c_{n+1}+b_n^2c_{n-1}
=(b_{n+1}+b_n)c_n\sim c_n,
\qquad
\frac{\|\tilde p_{n+1}\|_{v^{\alpha,\beta}}}{\sqrt{c_n}}\sim1.
\end{equation*}
We shall repeatedly use this norm equivalence throughout the paper.
\medskip
\begin{proposition}\label{prop:stimaptilde}
Let $I_n:=\left[-1+\C n^{-2},1-\C n^{-2}\right]$  and $u=v^{\gamma,\delta}$ with $\gamma,\delta\, \ge \, 0$. Assume that the parameters $\alpha$ and $\beta$ satisfy the strict  inequalities   in \eqref{cond1} and \eqref{cond2}.  Then
\begin{equation}\label{stimaptilde}
|\widehat p_{n+1}(x)u(x)|
\leq \C (1-x)^{-\frac{\alpha}{2}+\frac{1}{4}+\gamma} (1+x)^{-\frac{\beta}{2}+\frac{1}{4}+\delta},
\qquad x\in I_n,
\end{equation}
where $\C \neq \C(n)$.
 
Moreover,  for each $k=1,\dots,n+1$   one has
\begin{equation}\label{stimaderptilde}
\frac{1}{|\widehat{p}'_{n+1}(\tilde{x}_k)|u(\tilde{x}_k)}
\leq \C\, {\Delta} \tilde{x}_k \,
(1-\tilde{x}_k)^{\frac{\alpha}{2}-\frac{1}{4}-\gamma}(1+\tilde{x}_k)^{\frac{\beta}{2}-\frac{1}{4}-\delta},
\end{equation}
with  $\C \neq \C(n)$ and $\Delta \tilde{x}_k$  as in \eqref{defdistanza}.
\end{proposition}

\section{Lagrange Interpolation at the anti-Gauss nodes}\label{sec:interpolation}
Let us introduce the Lagrange polynomial interpolating a given function $f\in C_u$ at the $n+1$ zeros $\{\tilde{x}_k\}_{k=1}^{n+1}$ of the anti-Gauss polynomial $\tilde{p}_{n+1}$
\begin{equation}\label{polanti}
\tilde{L}_{n+1}(v^{\alpha,\beta},f,x)=\sum_{k=1}^{n+1} \tilde{\ell}_{k}(x) f(\tilde{x}_{k}),
\end{equation}
where the fundamental Lagrange polynomials are given by
\begin{equation}\label{polfon}
\tilde{\ell}_{k}(x)=\prod_{\substack{j=1 \\ j \neq k}}^{n+1}\frac{(x-\tilde{x}_{j})}{(\tilde{x}_{k}-\tilde{x}_{j})}=\frac{\tilde{p}_{n+1}(x)}{\tilde{p}'_{n+1}(\tilde{x}_k)(x-\tilde{x}_k)}.
\end{equation}
Note that such polynomials can be also defined in terms of a modified Christoffel-Darboux kernel~\cite[Theorem 2.2]{dfl2026} as follows 
\begin{equation}\label{fpanti}
\tilde \ell_k(x)=\tilde{\lambda}_k \, \tilde{K}_n(v^{\alpha,\beta},x,\tilde{x}_k),
\end{equation}
with
\begin{align*}
\tilde K_n(v^{\alpha,\beta},x,y)
&=
K_n(v^{\alpha,\beta},x,y)
-\frac12 p_n(v^{\alpha,\beta},x)p_n(v^{\alpha,\beta},y)\\
&=\frac12\left[K_n(v^{\alpha,\beta},x,y)+K_{n-1}(v^{\alpha,\beta},x,y) \right],
\end{align*}
where $K_n(v^{\alpha,\beta})$ is defined in \eqref{ChrisDarb}.

Let \begin{equation}\label{leb:functions}
\tilde{\Lambda}_{n+1}(x) :=  u(x) \sum_{k=1}^{n+1} \frac{|\tilde{\ell}_{k}(x)|}{u(x_{k})}
\end{equation}
the $(n+1)$-\emph{weighted Lebesgue function} associated to the zeros of $\tilde p_{n+1}.$

The norm of the  operator $\tilde{L}_{n+1}(v^{\alpha,\beta}):C_u \to C_u$ is given by 
\begin{align}
\label{leb_cost}\tilde{\Lambda}_{n+1}(v^{\alpha,\beta})& :=\left\|\tilde{L}_{n+1}(v^{\alpha,\beta})\right\|_{C_u\to C_u}  =\sup_{\|f\|_{C_u} \leq 1}\left\|\tilde{L}_{n+1}(v^{\alpha,\beta},f)\right\|_{C_u} \\ & =  \max_{x\in [-1,1]} \left\{ u(x) \sum_{k=1}^{n+1} \frac{|\tilde{\ell}_{k}(x)|}{u(x_{k})} \right\},\nonumber
\end{align}
and $\{\tilde{\Lambda}_{n+1}(v^{\alpha,\beta})\}_{n\ge0}$ is referred to as the sequence of \emph{weighted Lebesgue constants}.

It is well known that the behavior of  the sequence $\{\tilde{\Lambda}_{n+1}(v^{\alpha,\beta})\}_n$ plays  a crucial role in the study of the convergence of the interpolation sequence $\{\tilde L_{n+1}(v^{\alpha,\beta})\}_n$ in subspaces of $C_u$. Indeed,  for any $f\in C_u$,  the following estimate  holds
$$
\left\|f-\tilde L_{n+1}(v^{\alpha,\beta},f)\right\|_{C_u} \leq \tilde{\Lambda}_{n+1}(v^{\alpha,\beta}) \, E_n(f)_{u},
$$
with $E_n(f)_{u}$ defined in \eqref{best}. Consequently,  we first establish conditions under which  $\tilde{\Lambda}_{n+1}(v^{\alpha,\beta})=\mathcal{O}(\log n)$.
\vspace{0.5cm}
\begin{theorem}\label{th:stimaL}
Assume that the parameters $\alpha$ and $\beta$ satisfy the inequalities \eqref{cond1} and \eqref{cond2} strictly  and let $u=v^{\gamma,\delta}$ the weight defining the space $C_u$. If
\begin{equation}\label{condgammadelta}
\max\left\{0,\frac{\alpha}{2}-\frac{1}{4}\right\}
\leq\gamma\leq \frac{\alpha}{2}+\frac{3}{4},
\qquad
\max\left\{0,\frac{\beta}{2}-\frac{1}{4}\right\}
\leq\delta\leq \frac{\beta}{2}+\frac{3}{4},
\end{equation}
then,  we have
\begin{equation*}
\left\|\tilde{L}_{n+1}(v^{\alpha,\beta})\right\|_{C_u\to C_u}
\leq \C\log n,
\end{equation*}
where $\C \neq \C(n)$.
\end{theorem}

Next theorem is a direct consequence of the previous result.
\medskip
\begin{theorem}\label{th:convergence}
Assume that the hypotheses of Theorem~\ref{th:stimaL} are satisfied. Then, for every $f\in C_u$,
\[
\left\|f-\tilde L_{n+1}(v^{\alpha,\beta},f)\right\|_{C_u}
\leq
\C \log n\, E_n(f)_u,
\]
where $\C$ is a positive constant independent of $n$ and $f$. Moreover, for any $f \in W_r(u)$, $r \geq 1$, the following estimate holds 
\begin{equation}\label{estimateWr}
\left\|f-\tilde L_{n+1}(v^{\alpha,\beta},f)\right\|_{C_u} \leq \, \C \, \frac{\log n}{n^r}\, \|f\|_{W_r(u)},
\end{equation}
with $\C \neq \C(n,f)$.
\end{theorem}
\medskip
\begin{remark}
According to Theorem \ref{th:convergence}, the sequence of Lagrange interpolants
$\{\tilde L_{n+1}(v^{\alpha,\beta},f)\}_n$ exhibits the same approximation behavior as the sequence of best polynomial approximants in $C_u$, up to the logarithmic factor $\log n$.

A similar result holds for the  Lagrange interpolation process at the zeros
$\{x_k\}_{k=1}^{n}$ of the Jacobi polynomial $p_n(v^{\alpha,\beta})$, namely
\begin{equation}\label{pol}
\mathcal{L}_{n}(v^{\alpha,\beta},f,x)
=\sum_{k=1}^{n}\ell_k(x)f(x_k),
\qquad
\ell_k(x)=\lambda_kK_{n-1}(v^{\alpha,\beta},x,x_k),
\end{equation}
for which  the corresponding Lebesgue constants
\begin{equation}\label{leb:constants_gauss}
{\Lambda}_{n}(v^{\alpha,\beta})=\max_{x\in [-1,1]} u(x) \sum_{k=1}^{n} \frac{|{\ell}_{k}(x)|}{u(x_{k})}=\mathcal{O}(\log n)\end{equation} if and only if the following  conditions are satisfied  \cite{MR1997,Mastroianni08}
\begin{equation}\label{condgauss}
\max\left\{0,\frac{\alpha}{2}+\frac{1}{4}\right\}
\leq\gamma\leq\frac{\alpha}{2}+\frac{5}{4},
\qquad
\max\left\{0,\frac{\beta}{2}+\frac{1}{4}\right\}
\leq\delta\leq\frac{\beta}{2}+\frac{5}{4}.
\end{equation}

Comparing conditions \eqref{condgammadelta} and \eqref{condgauss}, we observe that the admissible ranges of $\gamma$ and $\delta$ ensuring optimality for the interpolation process \eqref{polanti} are shifted toward lower values. Consequently, the interpolation process  \eqref{polanti} is optimal for smaller values of the exponents defining the weight $u$.
In particular, if $\alpha\le 1/2$ and $\beta\le 1/2$, optimality is achieved even in the unweighted case, i.e., for $\gamma=\delta=0$, whereas the Lagrange interpolation process based on Jacobi nodes is not optimal. On the other hand, this shift toward lower values reduces the admissible range for larger values of $\gamma$ and $\delta$. Therefore, when stronger endpoint weights are needed to compensate for singularities of the interpolated function, the admissible ranges in \eqref{condgauss} become more favorable, and interpolation at Jacobi nodes may represent the more appropriate choice. This reflects a balance between the two interpolation processes: the proposed method enlarges the admissible range for smaller values of $\gamma$  and $\delta$, whereas interpolation at Jacobi nodes is more advantageous when larger values of the exponents of the  weight $u$ are required.
\end{remark}

\section{Numerical experiments}\label{sec:tests}
This section presents some  numerical experiments aimed to assess the performance of the interpolation process introduced in Section~\ref{sec:interpolation}.

The first numerical experiment investigates the behavior of the weighted Lebesgue functions and constants defined in \eqref{leb:functions} and \eqref{leb_cost}, respectively, in order to validate the estimates established in Theorem~\ref{th:stimaL}. It also compares these weighted Lebesgue functions and constants with the corresponding ones associated with the classical Lagrange interpolation process $\mathcal{L}_{n}(v^{\alpha,\beta})$ defined in \eqref{pol}.

The second experiment focuses on the approximation of three test functions, $f_i$, $i=1,2,3$, exhibiting different degrees of smoothness. For increasing values of $n$, we compute the weighted interpolation errors  
\begin{align}
\tilde{\epsilon}_n(f_i)
&=\max_{x\in X}
|[f_i(x)-\tilde{L}_{n+1}(v^{\alpha,\beta},f_i,x)]u(x)|, 
\qquad i=1,2,3, \label{erroreAG} \\
{\epsilon}_n(f_i)
&=\max_{x\in X}
|[f_i(x)-\mathcal{L}_{n+1}(v^{\alpha,\beta},f_i,x)]u(x)|,
\qquad i=1,2,3, \label{erroreG}
\end{align}
where $X$ is a sufficiently dense set of equispaced points in the interval $[-1,1]$. The computed errors $\tilde{\epsilon}_n(f_i)$ are compared with the theoretical estimates established in Theorem~\ref{th:convergence} and with the corresponding errors $\epsilon_n(f_i)$.

All the computations are performed on an Intel Xeon E-2244G system with 16Gb RAM, running Matlab R2024b.
\medskip

\begin{test}\label{tab:example1}
\rm
We first observe that for the Legendre weight ($\alpha=\beta=0$) and the second-kind Chebyshev weight ($\alpha=\beta=\tfrac12$), the left-hand sides of inequalities \eqref{cond1} and \eqref{cond2} reduce to $n^2+n$ when $\alpha=\beta=0$, and to $2n^2+4n+\tfrac32$ when $\alpha=\beta=\tfrac12$. Hence, inequalities \eqref{cond1} and \eqref{cond2} are strictly satisfied for every $n\ge1$. Consequently, in both cases, all anti-Gauss nodes lie in the open interval $(-1,1)$. We now investigate the behavior of the weighted Lebesgue constants \eqref{leb_cost} and \eqref{leb:constants_gauss} associated with the Legendre weight ($\alpha=\beta=0$) and the second-kind Chebyshev weight ($\alpha=\beta=\tfrac12$). In Figures~\ref{fig:test1} and \ref{fig:test2}, the orange dashed line represents the weighted anti-Gauss Lebesgue constants \eqref{leb_cost}, whereas the solid blue line represents the weighted Gauss Lebesgue constants \eqref{leb:constants_gauss}.

Figure~\ref{fig:test1} illustrates the behavior of $\tilde{\Lambda}_{n+1}(v^{\alpha,\beta})$ and $\Lambda_n(v^{\alpha,\beta})$  with $u(x)=(1-x^2)^{1/4}$ in the Legendre case and $u(x)=\sqrt{1-x^2}$ in the second-kind Chebyshev case. We note that the chosen weight functions $u$ satisfy both \eqref{condgammadelta} and \eqref{condgauss}. As expected, in both cases the weighted Lebesgue constants associated with the anti-Gauss interpolation process grow logarithmically, as established in Theorem~\ref{th:stimaL}, exhibiting the same asymptotic behavior as the weighted Lebesgue constants associated with the Gauss interpolation process.

\begin{figure}[ht]
\centering
\includegraphics[width=0.48\textwidth]{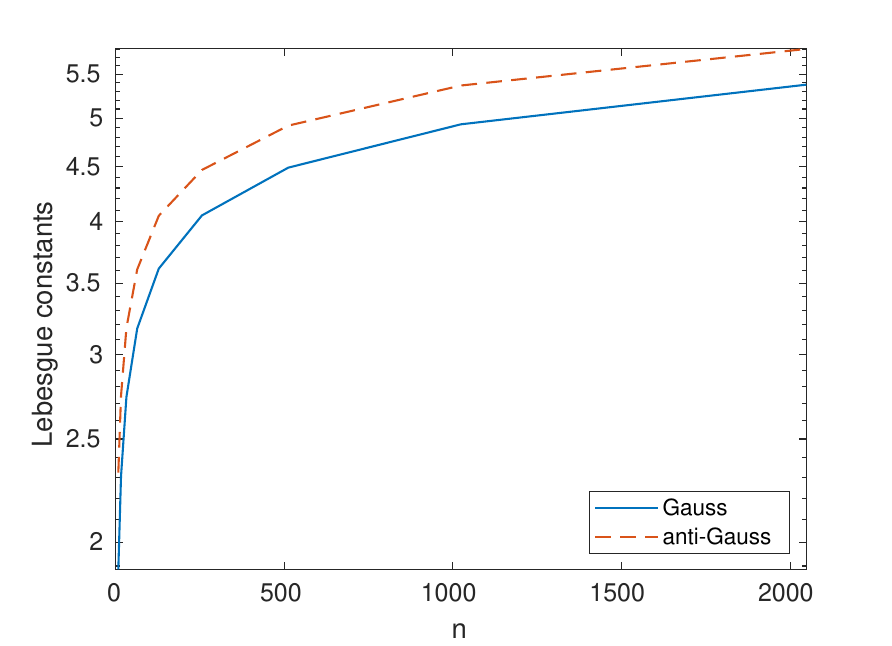}
\includegraphics[width=0.48\textwidth]{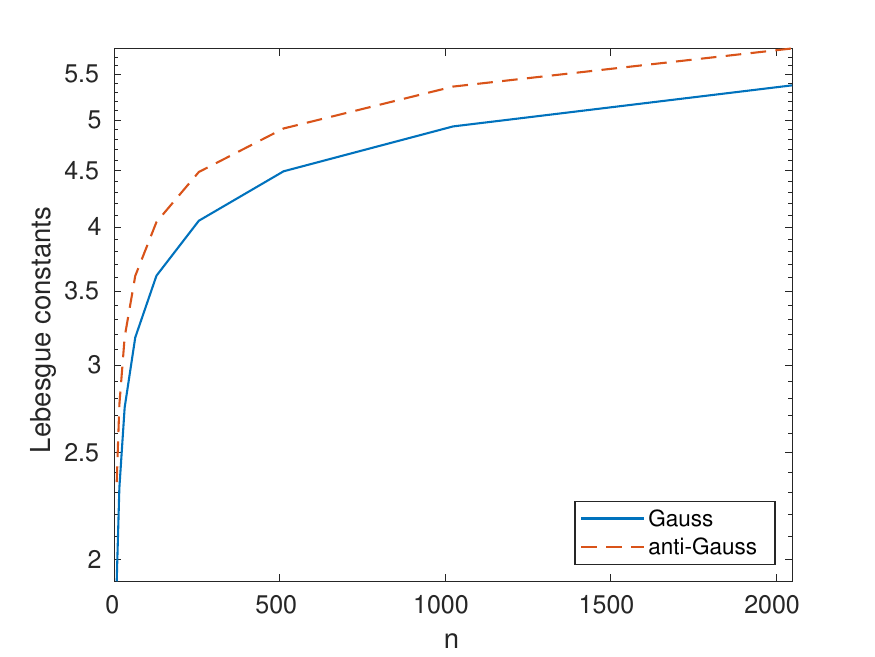}
\caption{To the left, the Lebesgue constants $\Lambda_n(v^{\alpha,\beta})$ and $\tilde{\Lambda}_{n+1}(v^{\alpha,\beta})$ associated to the Legendre weight $\alpha=\beta=0$ with $u(x)=\sqrt[4]{1-x^2}$. To the right, the Lebesgue constants $\Lambda_n(v^{\alpha,\beta})$ and $\tilde{\Lambda}_{n+1}(v^{\alpha,\beta})$ associated to the Chebyshev weight of second kind $\alpha=\beta=1/2$ with $u(x)=\sqrt{1-x^2}$.}
\label{fig:test1}
\end{figure}

We now investigate the behavior of the Lebesgue constants $\tilde{\Lambda}_{n+1}(v^{\alpha,\beta})$ and $\Lambda_n(v^{\alpha,\beta})$ still associated with the Legendre and second-kind Chebyshev weights,  in the unweighted case, i.e., for $\gamma=\delta=0$. Notice that the conditions in \eqref{condgammadelta} are still satisfied, whereas those in \eqref{condgauss} are not. Figure~\ref{fig:test2} displays the corresponding results. As predicted by conditions \eqref{condgauss} and \eqref{condgammadelta}, the Gauss Lebesgue constants exhibit a growth rate faster than logarithmic, whereas the anti-Gauss Lebesgue constants grow logarithmically.

\begin{figure}[ht]
\centering
\includegraphics[width=0.48\textwidth]{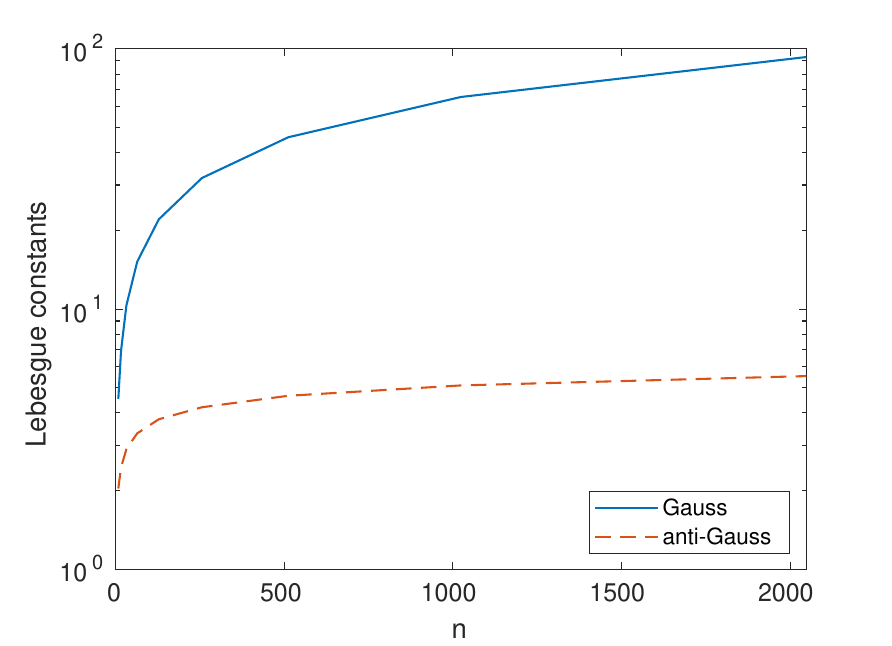}
\includegraphics[width=0.48\textwidth]{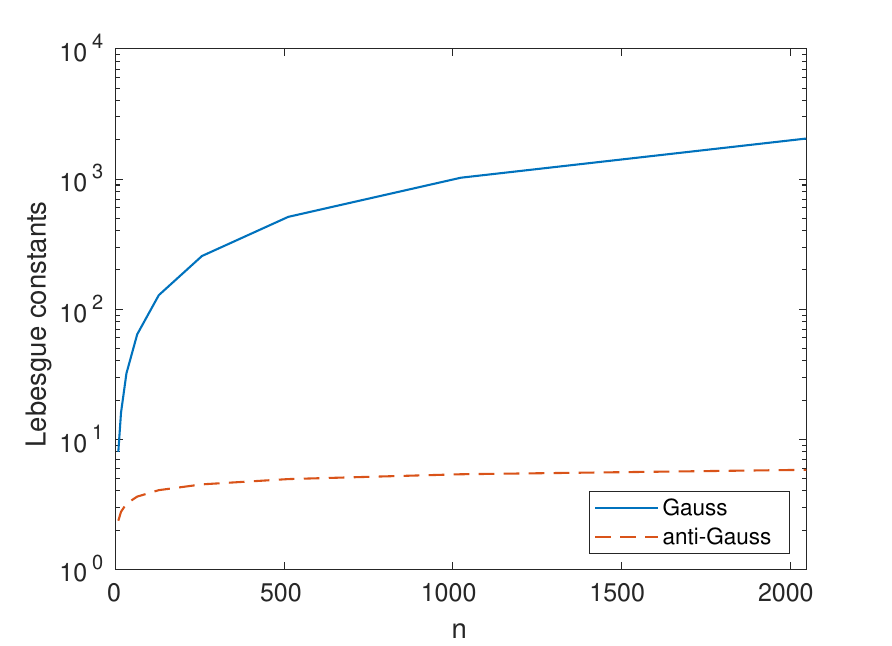}
\caption{To the left, the Lebesgue constants $\Lambda_n(v^{\alpha,\beta})$ and $\tilde{\Lambda}_{n+1}(v^{\alpha,\beta})$ associated to the Legendre weight ($\alpha=\beta=0$) with $u(x)=1$. To the right, those associated to the second-kind Chebyshev weight ($\alpha=\beta=1/2$)  with $u(x)=1$.}
\label{fig:test2}
\end{figure}

Finally, Figure~\ref{fig:test3} shows the Lebesgue functions associated with the weight $v^{1/4,1/2}$ for $n=32$ and $n=128$. Here, choosing $u(x)=(1-x)^{1/2}(1+x)^{3/4}$ ensures that both conditions \eqref{condgammadelta} and \eqref{condgauss} are satisfied. 
The plots show that the two Lebesgue functions exhibit essentially the same behavior. The anti-Gauss Lebesgue function is only slightly larger than the Gauss Lebesgue function over the entire interval.

\begin{figure}[ht]
\centering
\includegraphics[width=0.48\textwidth]{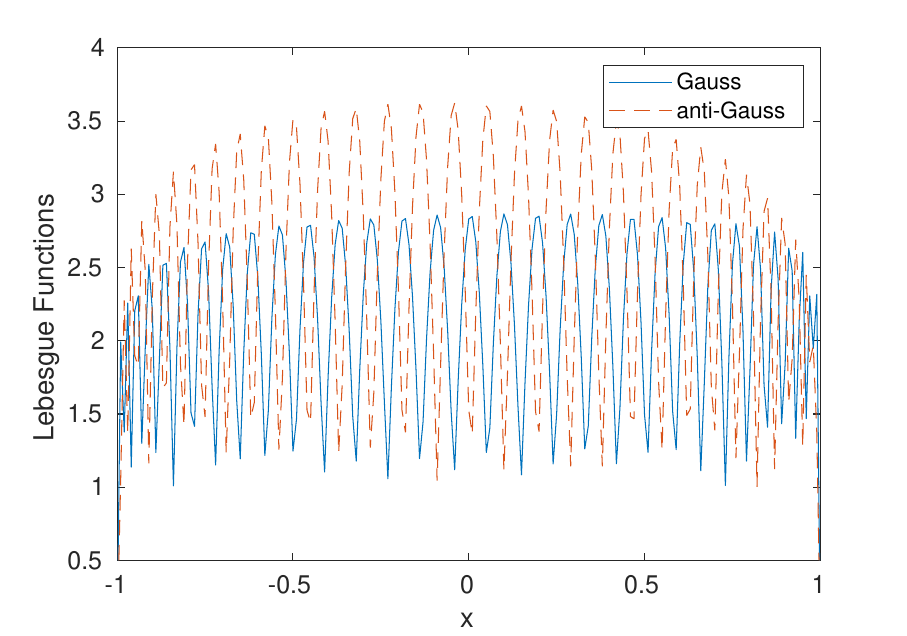}
\includegraphics[width=0.48\textwidth]{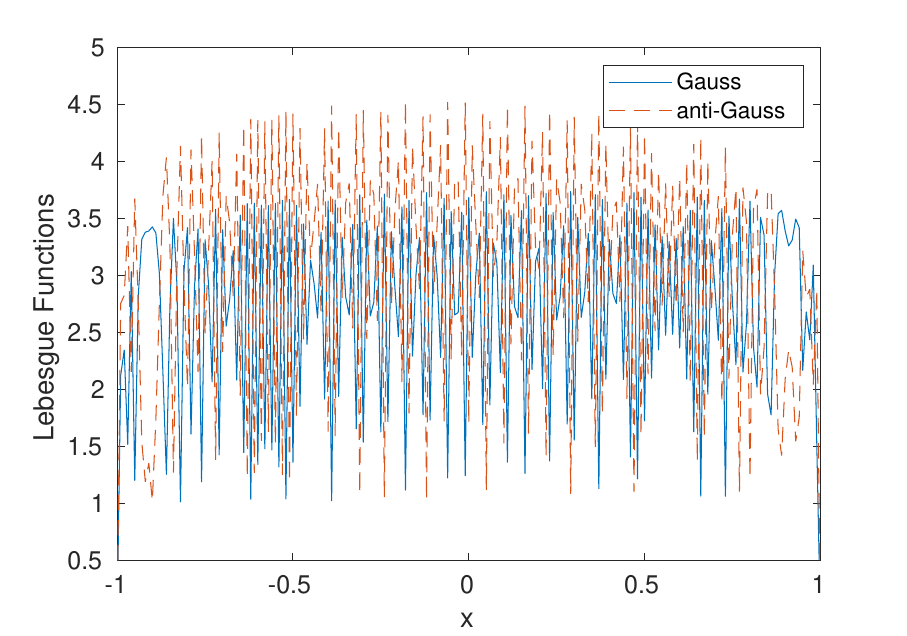}
\caption{The Lebesgue functions $\Lambda_n(v^{\alpha,\beta}, x)$ and $\tilde{\Lambda}_{n+1}(v^{\alpha,\beta}, x)$ for $\alpha=1/4$, $\beta=1/2$, $\gamma=1/2$ and $\delta=3/4$, with $n=32$ to the left and $n=128$ to the right. }
\label{fig:test3}
\end{figure}

\end{test}

\medskip

\begin{test}\label{tab:example2}
\rm Let us consider the following functions
\[
\begin{aligned}
f_1(x)&=(1+|x|)^{\frac{3}{2}}(x^2+\cos(x)),
& f_2(x)&=e^{-|x|^{\frac{5}{2}}},\\
f_3(x)&=|x+0.5|^{\frac{7}{2}}\frac{\sin(x)}{1+x^2}.
\end{aligned}
\]
and let us approximate them with the interpolating polynomials \eqref{polanti} and \eqref{pol} associated to the weights
\begin{equation*}
v^{\alpha,\beta}(x)=\begin{cases} 
(1+x)^{-1/3}, & i=1 \\
\sqrt{1-x^2}, & i=2 \\
(1-x)^{-1/4} (1+x)^{1/3}, & i=3 \\
\end{cases}.
\end{equation*}
Note that conditions \eqref{cond1} and \eqref{cond2} are satisfied, so that the nodes $\tilde{x}_k \in (-1,1)$ for all $k=1,\dots,n+1$. Table~\ref{table1} reports the errors \eqref{erroreAG} and ~\eqref{erroreG}
with a weight $u=v^{\gamma,\delta}$ satisfying both  \eqref{condgauss} and \eqref{condgammadelta}. In detail, we fix
\begin{equation*}
u(x)=\begin{cases} 
(1-x)^{1/4} , & i=1 \\
\sqrt{1-x^2} , & i=2 \\
(1-x)^{1/8} (1+x)^{5/12}, & i=3 
\end{cases}.
\end{equation*}

The numerical errors are smaller than the theoretical estimates provided by \eqref{estimateWr}. Since $f_1 \in W_1(u)$, $f_2 \in W_2(u)$, and $f_3 \in W_3(u)$, Theorem~\ref{th:convergence} predicts convergence rates of order $\mathcal{O}!\left(\log n, n^{-i}\right)$, $i=1,2,3$, respectively.  We emphasize that the computation for $n=1024$ and $n=2048$ are feasible thanks to the stable representations of the fundamental Lagrange polynomials $\ell_k$ and $\tilde\ell_k$ in terms of the Christoffel-Darboux kernel given in  \eqref{pol} and the modified Christoffel-Darboux in \eqref{fpanti}, respectively.

\begin{table}[ht]
\caption{Numerical Results of Test~\ref{tab:example2}}
\label{table1}
\centering
\small
\setlength{\tabcolsep}{4pt}
\begin{tabular}{c|cc|cc|cc}
$n$ & $\epsilon_n(f_1)$ & $\tilde{\epsilon}_n(f_1)$ & $\epsilon_n(f_2)$ & $\tilde{\epsilon}_n(f_2)$ &  $\epsilon_n(f_3)$ & $\tilde{\epsilon}_n(f_3)$  \\
\hline
16 & 8.90e-02 & 4.75e-02 & 1.01e-03 & 4.50e-04 & 2.95e-05 & 2.31e-05\\ 
  32 & 4.49e-02 & 2.41e-02 & 1.84e-04 & 8.47e-05 & 1.33e-06 & 2.76e-06\\ 
  64 & 2.25e-02 & 1.18e-02 & 3.34e-05 & 1.46e-05 & 2.26e-07 & 1.45e-07\\ 
 128 & 1.13e-02 & 5.92e-03 & 6.01e-06 & 2.58e-06 & 1.12e-08 & 2.14e-08 \\ 
 256 & 5.65e-03 & 1.78e-03 & 1.07e-06 & 3.62e-07 & 1.81e-09 & 5.38e-10\\ 
 512 & 2.83e-03 & 7.58e-04 & 1.90e-07 & 6.46e-08 & 4.41e-11 & 1.68e-10\\ 
 1024 & 1.41e-03 & 3.79e-04 & 3.38e-08 & 5.81e-09 & 1.43e-11 & 3.62e-12 \\ 
 2048 & 7.07e-04 & 1.90e-04 & 5.97e-09 & 7.61e-10  & 4.95e-12 & 1.30e-12 \\ 
\hline
\end{tabular}
\end{table} 
\end{test}

\section{Proofs}\label{sec:proofs}
This section gathers the proofs of the theoretical results presented in Section \ref{sec:newproperties} and \ref{sec:interpolation}.

In order to prove Proposition~\ref{prop:nodes1}, we need the following Lemma.

\begin{lemma}\label{lem:ag-structure}
The following identity holds true
\begin{equation}\label{ag-structure-identity}
\begin{aligned}
\frac{\tilde p_{n+1}(x)}{\sqrt{c_n}}
&=
-\frac{2}{2n+\alpha+\beta+1} (1-x^2)p_n'(v^{\alpha,\beta},x)\\
&\quad+
A_n(x)p_n(v^{\alpha,\beta},x), \qquad n \geq 1,
\end{aligned}
\end{equation}
where
\[
A_n(x):=
x-a_n+
\frac{2n(\alpha-\beta-(2n+\alpha+\beta)x)}{(2n+\alpha+\beta)(2n+\alpha+\beta+1)},
\] 
with $a_n$ defined in \eqref{aj}.
Furthermore,
\begin{equation}\label{ag-structure-derivative}
\frac{|\tilde p'_{n+1}(x)|}{\sqrt{c_n}}
\leq
\C  \left[n |p_n(v^{\alpha,\beta},x) |+\frac{1}{n}|p_n'(v^{\alpha,\beta},x) | \right],
\qquad |x|\leq 1,
\end{equation}
where $\C \neq \C(n)$.
\end{lemma}

\begin{proof}
First, let us write formula (2.3.17) in~\cite{Mastroianni08} in terms of the monic polynomial $p_n$
\begin{equation}\label{stima1}
p_{n-1}(x)=\frac{1}{g_n k_{n-1}} \left[(1-x^2) k_n p'_n(x) -(e_n x+f_n) k_n p_n(x)\right], 
\end{equation}
where~\cite[p. 132-133]{Mastroianni08}
\begin{equation*}
\begin{aligned}
e_n &=-n,
& f_n&=\frac{n(\alpha-\beta)}{2n+\alpha+\beta},\\
g_n&=\frac{2(n+\alpha)(n+\beta)}{2n+\alpha+\beta},
& k_n&=\frac{(n+\alpha+\beta+1)_n}{2^n n!}.
\end{aligned}
\end{equation*}

By replacing \eqref{stima1} in \eqref{recurrence} we have
\begin{align*}
p_{n+1}(x)
&=-\frac{b_n k_n}{g_n k_{n-1}} (1-x^2) p'_n(x)\\
&\quad+\left[ x-a_n+\frac{b_n k_n}{g_n k_{n-1}} (e_n x+f_n) \right] p_n(x)\\
&=-\frac{1}{2n+\alpha+\beta+1} (1-x^2) p'_n(x)\\
&\quad+\left[ x-a_n-\frac{nx(2n+\alpha+\beta)-n(\alpha-\beta)}
{(2n+\alpha+\beta+1)(2n+\alpha+\beta)} \right] p_n(x).
\end{align*}
Therefore, from \eqref{recurrence_anti1} we get
\begin{equation*}
\tilde{p}_{n+1}(x)=-\frac{2}{2n+\alpha+\beta+1} (1-x^2) p'_n(x)+A_n(x) p_n(x),
\end{equation*}
from which one deduce \eqref{ag-structure-identity}.

Let us now prove \eqref{ag-structure-derivative}. 
Differentiating \eqref{ag-structure-identity} we obtain
\begin{align*}
\frac{\tilde{p}'_{n+1}(x)}{\sqrt{c_n}}
&=-\frac{2}{2n+\alpha+\beta+1}
\left[(1-x^2)p_n'(v^{\alpha,\beta},x)\right]'\\
&\quad+A_n(x) p'_n(v^{\alpha,\beta},x)
+A'_n(x) p_n(v^{\alpha,\beta},x).
\end{align*} 
Hence, using~\cite{Mastroianni08}
\[
\begin{aligned}
\left[(1-x^2) p'_n(v^{\alpha,\beta},x)\right]'
&=
\left(\alpha-\beta+(\alpha+\beta)x\right)
p'_n(v^{\alpha,\beta},x)\\
&\quad-
n(n+\alpha+\beta+1)p_n(v^{\alpha,\beta},x),
\end{aligned}
\]
we get
\begin{align*}
\frac{\tilde{p}'_{n+1}(x)}{\sqrt{c_n}}&=\frac{2n(n+\alpha+\beta+1)+\alpha+\beta+1}{2n+\alpha+\beta+1} p_n(v^{\alpha,\beta},x) +S_n(x) p'_n(v^{\alpha,\beta},x),
\end{align*} 
where
$$S_n(x)=-a_n+\frac{(\alpha-\beta)(-2(2n+\alpha+\beta)-n)}{(2n+\alpha+\beta+1)(2n+\alpha+\beta)}+x \left[1-\frac{2(\alpha+\beta)+2n}{2n+\alpha+\beta+1} \right].$$
Therefore, being $|x|\leq 1$ and $a_n =\mathcal{O}(n^{-2})$, one deduces estimate \eqref{ag-structure-derivative}. 
\end{proof}

\bigskip
\begin{proof}[Proof of Proposition~\ref{prop:nodes1}]
To get the assertion, we need to prove
\begin{equation}\label{rel}
1+\tilde{x}_1 \sim \frac{\C}{n^2}
\qquad \textrm{and} \qquad
1-\tilde{x}_{n+1} \sim \frac{\C}{n^2}.
\end{equation}

First, the interlacing property and \eqref{gaussdist} give
$$
1+\tilde{x}_1<1+{x}_1 \leq \frac{\C}{n^2}, \quad
\text{and} \quad
1-\tilde{x}_{n+1}<1-{x}_{n} \leq \frac{\C}{n^2}, \qquad \C\neq\C(n).
$$
It remains to prove the reverse inequalities. By Lagrange's theorem,
\begin{equation*}
|\tilde{p}_{n+1}(-1)|=|\tilde{p}_{n+1}(-1)-\tilde{p}_{n+1}(\tilde{x}_1)|  =(1+\tilde{x}_1)|\tilde{p}_{n+1}'(\zeta)|,
\qquad
\zeta \in (-1,\tilde{x}_1).
\end{equation*}
Hence,
\begin{equation}\label{dist1}
1+\tilde{x}_1
=
\frac{|\tilde{p}_{n+1}(-1)|}{|\tilde{p}_{n+1}'(\zeta)|},
\qquad
\zeta \in (-1,\tilde{x}_1).
\end{equation}
From \eqref{valuepol_anti2}, \eqref{stimapol}, and switching to orthonormal polynomials, we have
\begin{equation}\label{valuepol_anti2_left}
|\tilde{p}_{n+1}(-1)|
\sim
\frac{|p_n(-1)|}{n}
=
\sqrt{c_n}\frac{|p_n(v^{\alpha,\beta},-1)|}{n}
\sim
\sqrt{c_n}n^{\beta-\frac12}.
\end{equation}
We now estimate the derivative. Lemma~\ref{lem:ag-structure} together with \eqref{stimapol} and \eqref{derpn} gives
\begin{equation}\label{der-left-nodes1}
|\tilde p'_{n+1}(\zeta)|
\leq
\C\sqrt{c_n}n^{\beta+\frac32}.
\end{equation}
Combining \eqref{dist1}, \eqref{valuepol_anti2_left}, and \eqref{der-left-nodes1}, we obtain
$$
1+\tilde{x}_1
\geq
\C\frac{\sqrt{c_n}n^{\beta-\frac12}}
{\sqrt{c_n}n^{\beta+\frac32}}
=
\frac{\C}{n^2},
$$
which proves the first relation in \eqref{rel}. The second relation can be proved analogously.
\end{proof}

\medskip
In order to prove Proposition~\ref{prop:nodes3}, we need the following result stated in a general setting in~\cite[Theorem 2.6]{GSCS2015}.
\medskip
\begin{theorem}\label{arcsindistr} 
Let $\{X_n\}$ and $\{Y_n\}$ be two sequences of matrix and define $\{Z_n\} = \{X_n + Y_n\}$. Assume that each $X_n$ is Hermitian and its eigenvalues $x_{n,k} \in (-1,1)$ for each $k=1,\dots,n$, are characterized by an asymptotic arc-sine distribution, i.e.,
\[
\displaystyle \lim_{n\to\infty} \left[\frac{1}{n} \sum_{k=1}^{n} \psi(x_{n,k})\right] = \frac{1}{\pi} \int_{-1}^{1} \frac{\psi(x)}{\sqrt{1 - x^2}} \, dx,
\] 
for each continuous function $\psi$.
If, in addition, the following conditions are fulfilled 
\begin{enumerate}
 \item[(a)] the spectral norms $\|X_n\|$ and $\|Y_n\|$ are uniformly bounded with respect to $n$,
 \item[(b)] the trace-norm of $Y_{n}$, i.e., the sum of the singular values of $Y_{n}$, is $o(n)$,
\end{enumerate}
then $\{Z_n\}$  has the same asymptotic eigenvalue distribution as $\{X_n\}$.
\end{theorem}
\medskip

\begin{proof}[Proof of Proposition~\ref{prop:nodes3}]
Let us prove the first part of the proposition. We write $\tilde{J}_{n+1}$ as a perturbation of the Jacobian matrix $J_{n+1}$ as follows
\[
\begin{aligned}
\tilde{J}_{n+1}
&=\begin{bmatrix}
J_n & \sqrt{b_n}\mathbf{e}_n \\
\sqrt{b_n} \mathbf{e}^T_n & a_n \\
\end{bmatrix}\\
&\quad+\begin{bmatrix}
\mathbf{0} & (\sqrt{2}-1)\sqrt{b_{n}}\mathbf{e}_n \\
(\sqrt{2}-1)\sqrt{b_{n}} \mathbf{e}^T_n & 0 \\
\end{bmatrix}\\
&=J_{n+1}+\Upsilon_{n+1}.
\end{aligned}
\]
where $\mathbf{0}$ represents the null matrix of order $n$.

It is well known  that the matrix $J_{n+1}$ is Hermitian and its eigenvalues tend to an arc-sine distribution as $n \to \infty$; see, for instance,~\cite{cs2019} and \eqref{gaussdist}. Moreover, $\Upsilon_{n+1}$ is also a Hermitian matrix. Therefore, for both $J_{n+1}$ and $\Upsilon_{n+1}$, the singular values $\sigma_k(J_{n+1})$ and $\sigma_k(\Upsilon_{n+1})$ coincide with the absolute values of the corresponding eigenvalues $\lambda_k(J_{n+1})$ and $\lambda_k(\Upsilon_{n+1})$, respectively, that is,
\[
\begin{aligned}
\sigma_k(J_{n+1})&=|\lambda_k(J_{n+1})|,\\
\sigma_k(\Upsilon_{n+1})&=|\lambda_k(\Upsilon_{n+1})|,
\qquad k=1,\ldots,n+1.
\end{aligned}
\]
At this point,  the assertion follows by applying Theorem~\ref{arcsindistr} after noting that assumptions $(a)$ and $(b)$ are satisfied as shown below.
\begin{enumerate}
\item[(a)] The spectral norms of $J_{n+1}$ and $\Upsilon_{n+1}$ are uniformly bounded with respect to $n$. Indeed, the eigenvalues of $J_{n+1}$ are the nodes $x_k$ of the $(n+1)$-Gauss formula, which always live in $(-1,1)$. Consequently, the spectral norm 
$$
\|J_{n+1}\|_2 =\max{|\lambda_k(J_{n+1})|}=\max{|x_k|}< 1, \quad k=1,\ldots,n+1,
$$ 
is uniformly bounded with respect to $n$.
Moreover, by construction, the eigenvalues of $\Upsilon_{n+1}$ are given by
$$\lambda_1(\Upsilon_{n+1})=\dots=\lambda_{n-1}(\Upsilon_{n+1})=0,$$
$$\lambda_n(\Upsilon_{n+1})=(\sqrt{2}-1)\beta_n, \quad \lambda_{n+1}(\Upsilon_{n+1})=-(\sqrt{2}-1)\beta_n.$$  Therefore, taking into account that $\beta_n\to \dfrac{1}{4}$ as $n\to \infty$,  we deduce the uniformly boundedness of $\|\Upsilon_{n+1}\|_2$ with respect to $n$
$$
\|\Upsilon_{n+1}\|_2= \max{|\lambda_k(\Upsilon_{n+1})|}=(\sqrt{2}-1)\beta_n \to \dfrac{(\sqrt{2}-1)}{4} \quad \text{as} \quad n\to\infty.
$$
\item[(b)] The trace-norm of $\Upsilon_{n+1}$ is given by 
$ \displaystyle 
\sum_k \sigma_k(\Upsilon_{n+1})=2(\sqrt{2}-1)\beta_n.
$
Consequently, the trace-norm of $ \Upsilon_{n+1}$ is $o(n)$ being 
$$
\frac{2(\sqrt{2}-1)\beta_n}{n} \to 0 \quad \text{as} \quad n\to\infty. 
$$
\end{enumerate}

Let us now demonstrate the second part of the proposition. The proof of estimate $\Delta \tilde x_k \leq \mathcal{C}\frac{\varphi(\tilde x_k)}{n}$ is reported in~\cite[Lemma 1]{DFR2020}. It remains to prove
$$
\Delta \tilde x_k \geq \mathcal{C}\frac{\varphi(\tilde x_k)}{n}.
$$
By applying Lagrange's theorem, we have
$$
|\tilde p_{n+1}(x_k)-\tilde p_{n+1}(\tilde x_k)|=(\tilde x_k-x_k)|\tilde p'_{n+1}(\zeta)|, \quad \zeta \in [\tilde x_k,x_k].
$$
Taking into account relation \eqref{interlacing} and the fact that $\tilde p_{n+1}(\tilde x_k)=0$, we can write
$$
\tilde x_{k+1}-\tilde x_k>\tilde x_k-x_k=\frac{|\tilde p_{n+1}(x_k)|}{|\tilde p'_{n+1}(\zeta)|}.
$$
Now, from the definition of $p_{n+1}(x)$ in \eqref{recurrence}, it is easy to see that $p_{n+1}(x_k)$ and $p_{n-1}(x_k)$ have opposite signs for $k=1,\ldots,n$. Combining this observation  with \eqref{recurrence_anti2}, we obtain
$$
|\tilde p_{n+1}(x_k)| \geq |p_{n+1}(x_k)| 
$$
and hence, by \eqref{stimapol}  
$$
|\tilde p_{n+1}(x_k)| \geq \mathcal{C} \sqrt{c_n} (1-x_k)^{-\frac{\alpha}{2}-\frac{1}{4}}(1+x_k)^{-\frac{\beta}{2}-\frac{1}{4}}.
$$
Moreover, estimate \eqref{der_anti-pol} gives
$$
\frac{1}{|\tilde p'_{n+1}(\zeta)|} \geq \mathcal{C} \sqrt{c_n} n (1-\zeta)^{-\frac{\alpha+1}{2}-\frac{1}{4}}(1+\zeta)^{-\frac{\beta+1}{2}-\frac{1}{4}}.
$$
Therefore,
$$
\tilde x_{k+1}-\tilde x_k \geq \mathcal{C} \frac{(1-x_k)^{-\frac{\alpha}{2}-\frac{1}{4}}(1+x_k)^{-\frac{\beta}{2}-\frac{1}{4}}}{n\,(1-\zeta)^{-\frac{\alpha+1}{2}-\frac{1}{4}}(1+\zeta)^{-\frac{\beta+1}{2}-\frac{1}{4}}}.
$$
Now, since~\cite[Lemma 1]{DFR2020}  
$$1 \pm x_{k-1} \sim 1 \pm \zeta \sim 1 \pm x_{k}, \quad \zeta \in [\tilde x_k,x_k],$$
and
$$(1-x^2_k)^\frac{1}{2} \sim (1-\tilde x^2_{k+1})^\frac{1}{2} \sim (1-\tilde x^2_{k})^\frac{1}{2}, \quad \tilde x_{k+1} \in [x_k,x_{k+1}],$$ 
we obtain
$$
\tilde x_{k+1}-\tilde x_k>\frac{\sqrt{1-x_k^2}}{n} \sim \frac{\sqrt{1-\tilde x_k^2}}{n},
$$
which proves the assertion.
\end{proof}
\medskip

\begin{proof}[Proof of Proposition~\ref{prop:stimaptilde}]
First, let us note that by combining \eqref{derpn} with \eqref{stimapol} we have
\begin{equation*}
\begin{aligned}
|p_n'(v^{\alpha,\beta},x)|
&\leq
\C \sqrt{n(n+\alpha+\beta+1)}
(1-x)^{-\frac{\alpha}{2}-\frac{3}{4}}\\
&\quad\times
(1+x)^{-\frac{\beta}{2}-\frac{3}{4}},
\qquad x \in I_n.
\end{aligned}
\end{equation*}
Therefore, identity \eqref{ag-structure-identity} yields
\[
\begin{aligned}
\frac{|\tilde p_{n+1}(x)|}{\sqrt{c_n}}
&\le
\C \left[
(1-x)^{-\frac{\alpha}{2}+\frac{1}{4}}
(1+x)^{-\frac{\beta}{2}+\frac{1}{4}}\right.\\
&\qquad\left.+
|A_n(x)|(1-x)^{-\frac{\alpha}{2}-\frac{1}{4}}
(1+x)^{-\frac{\beta}{2}-\frac{1}{4}}
\right].
\end{aligned}
\]
At this point, observing that
$$A_n(x) \leq \frac{\C}{n} \leq \sqrt{1-x^2},$$
being $x \in I_n$, and by multiplying by the weight $u$ we have

\[
\frac{|\tilde p_{n+1}(x)u(x)|}{\sqrt{c_n}}
\le
\C (1-x)^{-\frac{\alpha}{2}+\frac{1}{4}+\gamma} (1+x)^{-\frac{\beta}{2}+\frac{1}{4}+\delta},\]
i.e., estimate \eqref{stimaptilde}.

Let us now prove \eqref{stimaderptilde}. Using \eqref{fpanti} evaluated at \(x=\tilde x_k\), together with
\eqref{lambatilde2}, we have
\begin{equation}\label{kernel-diagonal-derivative-patch}
\tilde p'_{n+1}(\tilde{x}_k)= 2 \sqrt{c_n} \frac{ \tilde K_n(v^{\alpha,\beta},\tilde{x}_k,\tilde{x}_k)}{p_n(v^{\alpha,\beta}, \tilde{x}_k)}.
\end{equation}
Now, note that by using the positivity of the  kernel $\tilde K_n(v^{\alpha,\beta})$, we have the following bracketing
\begin{equation*}
\frac12 K_n(v^{\alpha,\beta},x,x)
\leq \tilde K_n(v^{\alpha,\beta},x,x)
\leq K_n(v^{\alpha,\beta},x,x).
\end{equation*}
Consequently, 
\begin{equation}\label{modified-jacobi-diagonal-kernel-estimate}
\tilde K_n(v^{\alpha,\beta},t,t)
\sim
\frac{n}{v^{\alpha,\beta}(t)\varphi(t)},
\qquad t\in I_n,
\end{equation} 
since by combining \eqref{estimatelambda} with \eqref{lambda} we have
\begin{equation*}
K_n(v^{\alpha,\beta},t,t)
\sim
\frac{n}{v^{\alpha,\beta}(t)\varphi(t)},
\qquad t\in I_n.
\end{equation*}
Therefore, we deduce
\begin{equation}\label{tilde-kernel-lower-patch}
\tilde K_n(v^{\alpha,\beta},\tilde{x}_k,\tilde{x}_k)
\ge
\C n (1-\tilde{x}_k)^{-\alpha-\frac{1}{2}} (1+\tilde{x}_k)^{-\beta-\frac{1}{2}}.
\end{equation}
Moreover, by \eqref{stimapol} we have
\begin{equation}\label{stima3}
|p_n(v^{\alpha,\beta},\tilde{x}_k)| \leq \C  (1-\tilde{x}_k)^{-\frac{\alpha}{2}-\frac{1}{4}} (1+\tilde{x}_k)^{-\frac{\beta}{2}-\frac{1}{4}}.
\end{equation}
Hence, by replacing \eqref{tilde-kernel-lower-patch} and  \eqref{stima3} in \eqref{kernel-diagonal-derivative-patch} we obtain for $k=1,\dots,n+1$,
\begin{align*} 
\frac{1}{|\widehat p'_{n+1}(\tilde{x}_k)|u(\tilde{x}_k)}
& \le
\frac{\C}{n}(1-\tilde{x}_k)^{\frac{\alpha}{2}+\frac{1}{4}-\gamma} (1+\tilde{x}_k)^{\frac{\beta}{2}+\frac{1}{4}-\delta} \nonumber  \\ &=
\C\frac{\sqrt{1-\tilde{x}_k}}{n}(1-\tilde{x}_k)^{\frac{\alpha}{2}-\frac{1}{4}-\gamma} (1+\tilde{x}_k)^{\frac{\beta}{2}-\frac{1}{4}-\delta}.
\end{align*}
Finally, taking \eqref{distanza} into account, we deduce \eqref{stimaderptilde}.
\end{proof}

\medskip
In order to prove Theorem~\ref{th:stimaL}, we use the following consequence of the arc-sine summation estimate reported in~\cite[Lemma~4.1.1, Remark~4.1.1, and formula~(4.1.13)]{Mastroianni08}.
For the sake of completeness, we present it here along with a brief outline of the proof.
\medskip
\begin{lemma}\label{lemma:somma}
Let \(N\ge 2\) and let $
-1 \equiv z_0<z_1<\cdots<z_N<z_{N+1} \equiv1$ a set of nodes such that 
be such that
\[
\begin{aligned}
1+z_1\sim \frac1{N^2},
&\qquad
1-z_N\sim \frac1{N^2},
\\
\Delta z_k = z_{k+1}-z_k \sim \frac{\varphi(z_k)}{N},
\qquad &k=1,\ldots,N.
\end{aligned}
\]
Let \(z\in I_N=:=\left[-1+\C N^{-2},1-\C N^{-2}\right]\), and let \(d\) be the index of a point \(z_d\) closest to \(z\), i.e., $
|z-z_d|=\displaystyle \min_{1\leq k\leq N}|z-z_k|.$
Then, for every \(A,B\in[0,1]\),
\[
\sum_{\substack{k=1\\ k\neq d}}^{N}
\left(\frac{1-z}{1-z_k}\right)^A
\left(\frac{1+z}{1+z_k}\right)^B
\frac{\Delta z_k}{|z-z_k|}
\leq
\C\log N,
\qquad
\C\neq \C(N,z).
\]
\end{lemma}

\begin{proof}
The result follows from~\cite[formula~(4.1.13)]{Mastroianni08} with
\(\gamma=-A\) and \(\delta=-B\), since \(A,B\in[0,1]\). Indeed, that formula yields
\[
\begin{aligned}
\sum_{\substack{k=1\\ k\neq d}}^{N}
\frac{\Delta z_k}{(1-z_k)^A\!(1+z_k)^B\!|z-z_k|}
&\le
\C
\left(\sqrt{1-z}+N^{-1}\right)^{-2A}
\left(\sqrt{1+z}+N^{-1}\right)^{-2B}\\
&\quad\times\log N .
\end{aligned}
\]
Multiplying by \((1-z)^A(1+z)^B\) and using
\[
(1-z)^A\leq \left(\sqrt{1-z}+N^{-1}\right)^{2A},
\qquad
(1+z)^B\leq \left(\sqrt{1+z}+N^{-1}\right)^{2B},
\]
gives the claimed estimate.
\end{proof}
\medskip

\begin{proof}[Proof of Theorem~\ref{th:stimaL}]
First, note that by Remez inequality~\cite{Mastroianni08} 
(see also~\cite{Fermo09})
\[
\left\|\tilde{L}_{n+1}(v^{\alpha,\beta},f)\right\|_{C_u}
\leq
\C
\max_{x \in I_n}
\left|
\tilde{L}_{n+1}(v^{\alpha,\beta},f,x)u(x)
\right|,
\qquad
\C \neq \C(n).
\]
Under assumption \eqref{condgammadelta}, one has
\[
0\leq \gamma-\frac{\alpha}{2}+\frac{1}{4}\leq1,
\qquad
0\leq \delta-\frac{\beta}{2}+\frac{1}{4}\leq1.
\]

Now, by definition \eqref{polanti}, and denoting by \(d\) the index of the node
\(\tilde x_d\) closest to \(x\), we have
\begin{align}\label{stima1L}
\left|
\tilde L_{n+1}(v^{\alpha,\beta},f,x)u(x)
\right|
&\leq
\|fu\|_\infty
\left[
\sum_{\substack{k=1\\ k\neq d}}^{n+1}
|\tilde \ell_k(x)|\frac{u(x)}{u(\tilde x_k)}
+
|\tilde \ell_d(x)|\frac{u(x)}{u(\tilde x_d)}
\right].
\end{align}

Let us estimate the sum at the right hand side of \eqref{stima1L}. Since the
fundamental polynomials are invariant under normalization, by \eqref{polfon} and
Proposition~\ref{prop:stimaptilde}, for \(x\in I_n\), we can write
\begin{align*}
\sum_{\substack{k=1\\ k\neq d}}^{n+1}
|\tilde \ell_k(x)|\frac{u(x)}{u(\tilde x_k)}
&=
\sum_{\substack{k=1\\ k\neq d}}^{n+1}
\frac{
|\widehat p_{n+1}(x)|u(x)
}{
|\widehat p'_{n+1}(\tilde x_k)|u(\tilde x_k)|x-\tilde x_k|
}
\\
&\leq
\C
\sum_{\substack{k=1\\ k\neq d}}^{n+1}
\left(\frac{1-x}{1-\tilde x_k}\right)^{-\frac{\alpha}{2}+\frac{1}{4}+\gamma}
\left(\frac{1+x}{1+\tilde x_k}\right)^{-\frac{\beta}{2}+\frac{1}{4}+\delta}
\frac{\Delta\tilde x_k}{|x-\tilde x_k|}.
\end{align*}

By Proposition~\ref{prop:nodes1} and Proposition~\ref{prop:nodes3}, the anti-Gauss nodes $\tilde x_k$ satisfy the assumptions of Lemma~\ref{lemma:somma}, with \(N=n+1\). Hence, under the assumption \eqref{condgammadelta}, we have
\[
\sum_{\substack{k=1\\ k\neq d}}^{n+1}
|\tilde \ell_k(x)|\frac{u(x)}{u(\tilde x_k)}
\leq
\C\log n.
\]

Let us now estimate the second term in \eqref{stima1L}. By the modified kernel
representation \eqref{fpanti} and by the identity
\(\tilde \ell_d(\tilde x_d)=1\), we have
\[
\tilde \ell_d(x)
=
\frac{
\tilde K_n(v^{\alpha,\beta},x,\tilde x_d)
}{
\tilde K_n(v^{\alpha,\beta},\tilde x_d,\tilde x_d)
}.
\]
Since \(\tilde K_n\) is positive definite, we have the Cauchy inequality
\[
|\tilde K_n(v^{\alpha,\beta},x,\tilde x_d)|^2
\leq
\tilde K_n(v^{\alpha,\beta},x,x)
\tilde K_n(v^{\alpha,\beta},\tilde x_d,\tilde x_d).
\]
Thus
\[
|\tilde \ell_d(x)|
\leq
\left(
\frac{
\tilde K_n(v^{\alpha,\beta},x,x)
}{
\tilde K_n(v^{\alpha,\beta},\tilde x_d,\tilde x_d)
}
\right)^{1/2}.
\]
Since $x,\tilde x_d\in I_n$, estimate \eqref{modified-jacobi-diagonal-kernel-estimate}
gives
\[
|\tilde \ell_d(x)|
\leq
\C
\left(
\frac{v^{\alpha,\beta}(\tilde x_d)\varphi(\tilde x_d)}
{v^{\alpha,\beta}(x)\varphi(x)}
\right)^{1/2}.
\]

Hence, since $1\pm x\sim 1\pm \tilde x_d,$ implies
\[
v^{\alpha,\beta}(x)\varphi(x)
\sim
v^{\alpha,\beta}(\tilde x_d)\varphi(\tilde x_d),
\qquad
u(x)\sim u(\tilde x_d),
\]
one deduce
\[
|\tilde \ell_d(x)|\frac{u(x)}{u(\tilde x_d)}
\leq
\C.
\]

Combining the two estimates in \eqref{stima1L}, we obtain
\[
\max_{x\in I_n}
\left|
\tilde L_{n+1}(v^{\alpha,\beta},f,x)u(x)
\right|
\leq
\C\log n\,\|f\|_{C_u}
\]
that proves the assertion.
\end{proof}

\medskip
\begin{proof}[Proof of Theorem~\ref{th:convergence}]
Let $P_n\in \mathbb{P}_n$ be the best polynomial approximation of $f$ in $C_u$. Since $\tilde L_{n+1}(v^{\alpha,\beta},P_n)=P_n$,
\begin{align*}
|[f(x)-\tilde L_{n+1}(v^{\alpha,\beta},f,x)]u(x)|  
&\leq |[f(x)-P_n(x)]u(x)|\\
&\quad+|\tilde L_{n+1}(v^{\alpha,\beta},P_n-f,x) u(x)|\\
&\leq \C (1 + \tilde{\Lambda}_{n+1}(v^{\alpha,\beta}))
E_{n}(f)_{u}. 
\end{align*}
Hence, by taking into account Theorem~\ref{th:stimaL} we get the assertion.
Finally, \eqref{estimateWr} follows by using estimate \eqref{stimaWr}.
\end{proof}

\section{Conclusions}\label{sec:conclusions}
In this paper, we investigated Lagrange interpolation processes based on the zeros of anti-Gauss Jacobi polynomials. Starting from the anti-Gauss polynomials associated with Jacobi weights, we established several new theoretical properties that are fundamental for approximation purposes.

In particular, assuming the anti-Gauss nodes are internal to the interval (-1,1), we proved that their distribution follows the arc-sine law and that the spacing between consecutive nodes exhibits the same asymptotic behavior as that of classical Gauss-Jacobi nodes. We also derived estimates for the normalized anti-Gauss Jacobi polynomials and their derivatives, providing the key ingredients required for the analysis of interpolation processes.

The obtained results allowed us to introduce a new Lagrange interpolation operator based on anti-Gauss nodes and  determine conditions ensuring that the weighted Lebesgue constants associated to the new  process behave with logarithmic growth. 
Consequently, the proposed interpolation process is optimal in the classical sense. An important feature of the proposed scheme is that  this optimality is achieved  for a different range of endpoint weight parameters than the classical Lagrange interpolation process based on Jacobi nodes. Taken together, the two interpolation schemes provide optimal Lebesgue constant estimates over a wider range of weighted spaces.

Convergence estimates were also established for functions belonging to Sobolev subspaces of weighted continuous functions spaces.

Numerical experiments confirm the theoretical findings.   The comparison with the classical Gauss-Jacobi interpolation process shows  that the anti-Gauss based interpolant provides a competitive and, in several cases, more accurate approximation. 

Future research may concern the extension of these results to other families of anti-Gauss orthogonal polynomials, the construction of multidimensional interpolation schemes, and the development of efficient numerical methods for integral and functional equations based on the interpolation processes introduced in this work.

\section*{Acknowledgements}
This research has been accomplished within “Research ITalian network on Approximation” (RITA). The authors are members of the Gruppo Nazionale Calcolo Scientifico-Istituto Nazionale di Alta Matematica (GNCS-INdAM).
L. Fermo and D. Occorsio are also members of the TAA-UMI Research Group, and the SIMAI Activity Group ``Numerical and Analytical Approximation of Data and Functions with Applications''.

P. D\'iaz de Alba, L. Fermo, and D. Occorsio are partially supported by INdAM-GNCS 2026 project ``Metodi numerici per modelli integrali e dinamiche con memoria''. L. Fermo is also partially supported by Fondazione di Sardegna biennial project 2024-2025 ``Integral and Discrete Inverse Problems (InDIP)''. V. Loi is partially supported by INdAM-GNCS 2026 project ``Metodi strutturati per il signal processing avanzato''.

\bibliographystyle{amsplain}
\bibliography{biblio}

\end{document}